\documentclass{amsart} 

\usepackage{amsmath,amsthm,amssymb,amscd}
\usepackage{mathrsfs,bm,color,graphicx,subdepth}
\usepackage[colorlinks=true, linkcolor=magenta, citecolor=cyan, urlcolor=blue]{hyperref}

\theoremstyle{plain}
\newtheorem{thm}{Theorem}[section]
\newtheorem{prop}[thm]{Proposition}
\newtheorem{coro}[thm]{Corollary}
\newtheorem{lem}[thm]{Lemma}
\newtheorem{obs}[thm]{Observation}
\newtheorem{thmalpha}{Theorem}
\newtheorem{lemalpha}[thmalpha]{Lemma}
\newtheorem{obsalpha}[thmalpha]{Observation}

\newtheorem*{thm-Closed}{\hyperref[thm:Closed]{Theorem~\ref*{thm:Closed}}}
\newtheorem*{thm-Open}{\hyperref[thm:Open]{Theorem~\ref{thm:Open}}}
\newtheorem*{thm-CSF}{\hyperref[thm:CSF]{Theorem~\ref{thm:CSF}}}
\newtheorem*{thm-Ess}{\hyperref[thm:Ess]{Theorem~\ref{thm:Ess}}}
\theoremstyle{definition}

\theoremstyle{remark}
\newtheorem*{rem}{Remark}

\newcommand{\ts}{\hspace{.11111em}}
\newcommand{\tts}{\hspace{.05555em}}
\newcommand{\nts}{\hspace{-.11111em}}
\newcommand{\ntts}{\hspace{-.05555em}}

\newcommand{\bbZ}{\mathbb{Z}}
\newcommand{\bbA}{\mathbb{A}}

\newcommand{\bsM}{{\boldsymbol{M}}}
\newcommand{\bsZ}{{\boldsymbol{Z}}}
\newcommand{\bsX}{{\boldsymbol{X}}}
\newcommand{\bsE}{{\boldsymbol{E}}}
\newcommand{\bsS}{{\boldsymbol{S}}}
\newcommand{\bsH}{{\boldsymbol{H}}}
\newcommand{\bsP}{{\boldsymbol{P}}}
\newcommand{\bsF}{{\boldsymbol{F}}}
\newcommand{\bsT}{{\boldsymbol{T}}}
\newcommand{\bsw}{{\boldsymbol{w}}}
\newcommand{\bszero}{{\boldsymbol{0}}}
\newcommand{\bszeta}{{\boldsymbol{\zeta}}}
\newcommand{\bsvarphi}{{\boldsymbol{\varphi}}}
\newcommand{\bssigma}{{\boldsymbol{\sigma}}}
\newcommand{\bsgamma}{{\boldsymbol{\gamma}}}

\newcommand\fakeslant[1]{\pdfliteral{1 0 0.25 1 0 0 cm}#1\pdfliteral{1 0 -0.25 1 0 0 cm}}
\newcommand{\mathbbs}[1]{\mathbb{\fakeslant{#1}}}
\newcommand{\bbsR}{\mathbbs{R}\tts}

\newcommand{\nil}{\varnothing}
\newcommand{\ssminus}{\smallsetminus}

\newcommand{\bbC}{\mathbb{C}}
\newcommand{\bbZtwo}{\mathbb{Z}_{\nts/ 2\mathbb{Z}}}

\newcommand{\bbZp}{\mathbb{Z}_{\nts/ \ntts p\mathbb{Z}}}

\newcommand{\Sys}{\operatorname{\mathsf{Sys}}\tts}
\newcommand{\Vol}{\operatorname{\mathsf{Vol}}\tts}
\newcommand{\Area}{\operatorname{\mathsf{Area}}\tts}
\newcommand{\Length}{\operatorname{\mathsf{Length}}\tts}

\newcommand{\Sph}{\bsS} \newcommand{\sph}{\mathsf{sph}}
\newcommand{\Euc}{\bsE} \newcommand{\euc}{\mathsf{euc}}
\newcommand{\Hyp}{\bsH} 
\newcommand{\Geo}{\bsX}
\newcommand{\SR}{\bsS^2 \times \bbsR}
\newcommand{\HR}{\bsH^2 \times \bbsR}
\newcommand{\Sol}{\boldsymbol{S} \ntts \boldsymbol{o} \tts \boldsymbol{l}}
\newcommand{\Nil}{\boldsymbol{N} \nts \boldsymbol{i} \tts \boldsymbol{l}}
\newcommand{\SLtwo}{\boldsymbol{S} \nts \boldsymbol{L}^{\raise.2ex\hbox{$\scriptstyle\bm{\sim}$}}_2}

\newcommand{\SL}{\mathit{SL}}
\newcommand{\Homeo}{\mathsf{Homeo}}
\newcommand{\Mod}{\mathsf{Mod}}

\newcommand{\interior}{\operatorname{int}\tts}

\newcommand{\csum}{\mathbin{\#}}

\newcommand{\RP}{\bbsR \ntts P}
\newcommand{\RProj}{\bbsR \ntts \bsP}

\newcommand{\PD}{\mathit{PD}}
\newcommand{\PLD}{\mathit{LD}}
\newcommand{\Ksupp}{{\mathcal{K}}}
\newcommand{\Fsupp}{{\mathcal{F}}}

\begin{document}


\title[On Isosystolic Inequalities for $T^n$, $\RP^n$, and $M^3$]{On Isosystolic Inequalities for $\bsT^n$, $\RProj^n$, and $\bsM^3$}
\author[Nakamura]{Kei Nakamura}
\address{Department of Mathematics\\ Temple University \\ Philadelphia, PA 19106}
\curraddr{Department of Mathematics\\ University of California \\ Davis, CA 95616}
\email{knakamura@math.ucdavis.edu}

\begin{abstract}
If a closed smooth $n$-manifold $M$ admits a finite cover $\hat M$ whose $\bbZtwo$-cohomology has the maximal cup-length, then for any riemannian metric $g$ on $M$, we show that the systole $\Sys(M,g)$ and the volume $\Vol(M,g)$ of the riemannian manifold $(M,g)$ are related by the following \emph{isosystolic inequality}:
\begin{align*}
\Sys(M,g)^n & \leq n! \Vol(M,g).
\end{align*}
The inequality can be regarded as a generalization of Burago and Hebda's inequality for closed essential surfaces and as a refinement of Guth's inequality for closed $n$-manifolds whose $\bbZtwo$-cohomology has the maximal cup-length. We also establish the same inequality in the context of possibly non-compact manifolds under a similar cohomological condition. The inequality applies to (i) $T^n$ and all other compact euclidean space forms, (ii) $\RP^n$ and many other spherical space forms including the Poincar\'e dodecahedral space, and (iii) most closed essential 3-manifolds including all closed aspherical 3-manifolds.
\end{abstract}

\maketitle
\setcounter{tocdepth}{1}


\section{Introduction} \label{sec:Intro}

For a complete riemannian $n$-manifold $(M, g)$, the \emph{systole} $\Sys(M,g)$ is the infimum of the length of non-contractible loops in $(M,g)$; when $M$ is a closed manifold, the systole is realized by the shortest non-contractible geodesic loop. For a fixed underlying smooth $n$-manifold $M$, an \emph{isosystolic inequality} for $M$ takes the form
\[
\Sys(M,g)^n \leq C \Vol(M,g)
\]
where $C$ is a constant, independent of a metric $g$ on $M$. Such a constant $C$ will be referred to as an \emph{isosystolic constant} for $M$.

\subsection{Essential Surfaces} \label{ssec:IntroInequality}

Isosystolic inequalities have been studied extensively for closed surfaces. We mention a few results in particular; see \cite[\S1]{Croke-Katz:Survey} for a brief survey, or \cite[Part~1]{Katz:Book} for a broader overview. The study of isosystolic inequalities goes back to the following remarkable theorems by Loewner and Pu \cite{Pu}. 

\begin{thmalpha}[Loewner] \label{thm:Loewner}
For any riemannian metric $g$ on the 2-torus $T^2 \nts$,
\[
\Sys(T^2\nts,g)^2 \leq \frac{2}{\sqrt{3}} \Vol(T^2\nts,g).
\]
Moreover, $2/\sqrt{3}$ is the optimal isosystolic constant for $T^2\nts$.
\end{thmalpha}

\begin{thmalpha}[Pu] \label{thm:Pu}
For any riemannian metric $g$ on the real projective plane $\RP^2\nts$,
\[
\Sys(\RP^2\nts,g)^2 \leq \frac{\pi}{2} \Vol(\RP^2\nts,g).
\]
Moreover, $\pi/2$ is the optimal isosystoic constant for $\RP^2\nts$.
\end{thmalpha}

For the Klein bottle $\RP^2\csum\RP^2 \nts$, Bavard \cite{Bavard:Klein} established an isosystolic inequality with the optimal isosystolic constant $\pi/(2\sqrt{2})$. These three surfaces $T^2 \nts$, $\RP^2 \nts$, and $\RP^2\csum\RP^2\nts$, are the only manifolds in any dimension for which the optimal isosystolic constants are known; any ``optimal'' isosystolic inequalities for other manifolds in the literature, e.g. \cite{Katz-Sabourau:CAT0}, require restrictions on the class of metrics in consideration. Generally, identifying the optimal isosystolic constant for any manifold is notoriously difficult. 

A closed surface is said to be \emph{essential} if it is not $S^2$. For all essential surfaces, some isosystolic inequality holds. For orientable surfaces, an inequality with an isosystolic constant $C=2$ appeared implicitly in the work of Besicovitch \cite{Besicovitch} and Rodin \cite{Rodin}, and explicitly in the work of Burago \cite[Ch.1 \S5]{Burago-Zalgaller} and Hebda \cite{Hebda}; as Gromov noted \cite[\S5.1]{Gromov:Filling}, Hebda's proof applies to non-orientable surfaces with little modification. For aspherical surfaces, Gromov improved the inequality \cite[\S5.2]{Gromov:Filling} with an isosystolic constant $C=4/3$. Hence, $C=\pi/2$ is the optimal \emph{universal} isosystolic constant for essential surfaces, by the work of Pu ($\pi/2$ is optimal for $\RP^2$) and Gromov ($4/3$ suffices for aspherical surfaces).

\subsection{Essential Manifolds}

The underlying reason for the existence of isosystolic inequalities is in the relationship between the geometry and the topology of the manifold. Gromov addressed this issue with great insights and established his celebrated isosystolic inequality for all \emph{essential} manifolds \cite[\S0.1]{Gromov:Filling} in all dimensions; here, a closed manifold $M$ is said to be \emph{essential} if there is an aspherical space $K$ and a map $f:M \rightarrow K$ such that $f$ represents a non-zero homology class in $H_n(K;\bbA)$, with $\bbA=\bbZ$ if $M$ is orientable and $\bbA=\bbZtwo$ if $M$ is non-orientable.

\begin{thmalpha}[Gromov] \label{thm:Gromov}
For each dimension $n$, there exists a constant $C_n$ such that, for any closed essential $n$-manifold $M$ and any riemannian metric $g$ on $M$,
\begin{align*}
\Sys(M,g)^n & \leq C_n \Vol(M,g).
\end{align*}
Explicitly, one may take $C_n=\big(\ts 6(n+1) \ntts\cdot n^n \ntts\cdot\ntts \sqrt{(n+1)!}\; \big)^n$.
\end{thmalpha}

Essential manifolds generalize essential surfaces. Every aspherical manifold $M \simeq K(\pi_1(M),1)$, such as the $n$-torus $T^n$, is essential via the identity map; the real projective $n$-space $\RP^n$ is essential via the natural inclusion map $\RP^n \hookrightarrow \RP^\infty \simeq K(\pi_1(\RP^n\ntts),1)$. Babenko \cite{Babenko:Asymptotic} showed that an orientable manifold $M$ is essential if and only if some isosystolic inequality holds for $M$.

The universal constant $C_n$ in \hyperref[thm:Gromov]{Theorem~\ref*{thm:Gromov}} is quite large even for small $n$, and grows rapidly with respect to $n$; it is hardly optimal for any individual manifold, and it is not meant to be optimal as a universal constant for essential $n$-manifolds. For dimension 2, Gromov's universal constant is $C_2=31,\!104$ although we know that the optimal universal constant for essential surfaces is $\pi/2$. In dimension 3, Gromov's universal constant works out to be $C_3 \approx 31,\!992,\!036,\!026$ or roughly 32 billion. 
Wenger showed in \cite{Wenger} that one can take the universal constant for essential $n$-manifolds to be $C_n=(\tts 6 \cdot 27^n \cdot (n+1)! \tts)^n$ which has a substantially slower growth rate than Gromov's constant; it should be noted, however, that $n$ needs to be quite large for Wenger's constant to be smaller than Gromov's constant.

\subsection{Maximal Cup-Length}

While finding the optimal isosystolic constant is extremely difficult for any individual manifold in general, one hopes to find a better isosystolic inequality for some manifold or for some class of manifolds. It is known from Gromov's work \cite[\S3.C]{Gromov:Intersystolic} that isosystolic inequalties can be deduced from the non-degeneracy of cup products. Using \emph{nearly minimizing} hypersurfaces as the primary tool in the proof, Guth established in \cite{Guth:Hypersurfaces} an isosystolic inequality for any closed manifold whose $\bbZtwo$-cohomology has the maximal cup-length.

\begin{thmalpha}[Guth] \label{thm:Guth}
Let $M$ be a closed smooth $n$-manifold. Suppose that there exist (not necessarily distinct) cohomology classes $\bszeta_1, \cdots, \bszeta_n$ in $H^1(M;\bbZtwo)$ such that $\bszeta_1 \cup \cdots \cup \bszeta_n \neq \bszero$ in $H^n(M;\bbZtwo)$. Then, for any riemannian metric $g$ on $M$,
\begin{align*}
\Sys(M,g)^n & \leq (8n)^n \Vol(M,g).
\end{align*}
\end{thmalpha}

Manifolds satisfying the hypothesis of Guth's theorem include many important essential manifolds, such as tori $T^n$ and real projective spaces $\RP^n$ in all dimensions, as well as all essential surfaces. In particular, \hyperref[thm:Guth]{Theorem~\ref*{thm:Guth}} provides the best previously known isosystolic constant $C_n=(8n)^n$ for $T^n$ and $\RP^n$ with $n \geq 3$. It should be emphasized that Guth's constant is significantly smaller and grows substantially more slowly than the universal constant given by Gromov or by Wenger. We record an improved inequality in \hyperref[sec:Closed]{\S\ref*{sec:Closed}}:

\begin{thm-Closed} \label{thm:Closed-restated}
Let $M$ be a closed smooth $n$-manifold, and let $\hat M$ be a finite-degree cover of $M$. Suppose that there exist (not necessarily distinct) cohomology classes $\bszeta_1, \cdots, \bszeta_n$ in $H^1(\hat M;\bbZtwo)$ such that $\bszeta_1 \cup \cdots \cup \bszeta_n \neq \bszero$ in $H^n(\hat M;\bbZtwo)$.
Then, for any riemannian metric $g$ on $M$,
\begin{align*}
\Sys(M,g)^n & \leq n! \Vol(M,g).
\end{align*}
\end{thm-Closed}

Our inequality can be regarded as a higher-dimensional generalization of Hebda's (with $C_2=2$ in dimension $n=2$), as well as a refinement of Guth's (with $C_n=(8n)^n$ in any dimension $n$). In order to establish the inequality, we adapt Guth's approach \cite{Guth:Hypersurfaces} and utilize nearly minimizing hypersurfaces. We note that our inequality holds for a broader class of manifolds in comparison to \hyperref[thm:Guth]{Theorem~\ref*{thm:Guth}}; this is achieved with a covering argument in \cite{KKSSW}.

We also prove in \hyperref[sec:Open]{\S\ref*{sec:Open}} a more general version of \hyperref[thm:Closed-restated]{Theorem~\ref*{thm:Closed}} for manifolds that are possibly non-compact. In this general context, our main result can be stated in terms of  cohomology $H^*_\Ksupp(M;\bbZtwo)$ with \emph{compact support} and ``ordinary'' cohomology $H^*_\Fsupp(M;\bbZtwo) = H^*(M;\bbZtwo)$ with closed support as follows.

\begin{thm-Open} \label{thm:Open-restated}
Let $M$ be a (not necessarily closed) smooth $n$-manifold, and let $\hat M$ be a (possibly infinite-degree) cover of $M$. Suppose that there exist (not necessarily distinct) cohomology classes $\bszeta_1, \cdots, \bszeta_n$ in $H^1_\Fsupp(\hat M;\bbZtwo)$ or $H^1_\Ksupp(\hat M;\bbZtwo)$, with at least one of them in $H^1_\Ksupp(\hat M;\bbZtwo)$, such that $\bszeta_1 \cup \cdots \cup \bszeta_n \neq \bszero$ in $H^n_\Ksupp(\hat M;\bbZtwo)$. Then, for any complete riemannian metric $g$ on $M$,
\begin{align*}
\Sys(M,g)^n &\leq n! \Vol(M,g).
\end{align*}
\end{thm-Open}

\subsection{Applications}

A large portion of this article is devoted to the presentation of examples of closed manifolds, to which \hyperref[thm:Closed-restated]{Theorem~\ref*{thm:Closed}} or \hyperref[thm:Open-restated]{Theorem~\ref*{thm:Open}} can be applied. We establish the best known isosystolic inequality in many cases.

In \hyperref[sec:TRP]{\S\ref*{sec:TRP}}, we consider \emph{compact space forms} with non-negative sectional curvature. The primary examples of closed manifolds whose $\bbZtwo$-cohomology have the maximal cup-length are the $n$-torus $T^n$ and the real projective $n$-space $\RP^n$ for all $n$, and \hyperref[thm:Closed-restated]{Theorem~\ref*{thm:Closed}} establishes isosystolic constants $C_n=n!$ for them; for $n \geq 3$, our constant improves upon the best previously known constant $C_n=(8n)^n$ by Guth (\hyperref[thm:Guth]{Theorem~\ref*{thm:Guth}}). More generally, \hyperref[thm:Closed-restated]{Theorem~\ref*{thm:Closed}} applies to all compact euclidean space forms and many spherical space forms.

\begin{thm-CSF} \label{thm:CSF-restated}
Let $M$ be a closed $n$-manifold. Suppose that $M$ is homeomorphic to either (i) a compact euclidean space form, such as the $n$-torus $T^n\nts$, or (ii) a spherical space form with even-order fundamental group, such as the real projective $n$-space $\RP^n\nts$. Then, for any riemannian metric $g$ on $M$,
\begin{align*}
\Sys(M,g)^n &\leq n! \, \Vol(M,g).
\end{align*}
\end{thm-CSF}

In \hyperref[sec:Geom3]{\S\ref*{sec:Geom3}}, \hyperref[sec:Asph3]{\S\ref*{sec:Asph3}}, and \hyperref[sec:Ess3]{\S\ref*{sec:Ess3}}, we discuss applications of our results to closed 3-manifolds. For many of these 3-manifolds, the best previously known isosystolic constant was either Gromov's ($C_3 \approx 32$ billion) or Guth's ($C_3 = 13,\!842$). It turns out that \hyperref[thm:Open-restated]{Theorem~\ref*{thm:Open}} applies to most closed 3-manifolds, including \emph{all} closed aspherical 3-manifolds, and establishes an isosystolic constant $C_3=6$ for them. Let us write $V_0:=S^3\nts$, $V_k:=\csum_k(S^2 \times S^1\ntts)$ and $V_{-k}:=\csum_k(S^2 \rtimes S^1\ntts)$ for $k>0$, where $\csum_k$ denotes the $k$-times repeated connected sum and $S^2 \rtimes S^1$ denotes the non-orientable $S^2$-bundle over $S^1$. Also, for coprime positive integers $p > q$, $L(p,q)$ denotes the lens space obtained as a quotient of $S^3 \subset \bbC^2$ by the $\bbZp$-action generated by $(z_1,z_2) \mapsto \big(e^{2\pi i q/p}z_1,e^{2\pi i q/p}z_2 \big)$; the integer $p$ is referred to as the \emph{order} of $L(p,q)$. Our main result for closed 3-manifolds is the following.

\begin{thm-Ess} \label{thm:Ess-restated}
Let $M$ be a closed 3-manifold. Suppose that $M$ is not a connected sum $L(p_1,q_1)\csum \cdots \csum L(p_m,q_m) \csum V_k$ with odd orders $p_1, \cdots, p_m$ and an integer $k$. Then, for any riemannian metric $g$ on $M$,
\begin{align*}
\Sys(M,g)^3 &\leq 6 \, \Vol(M,g).
\end{align*}
\end{thm-Ess}

\hyperref[thm:Ess-restated]{Theorem~\ref*{thm:Ess}} relies heavily and essentially on some of the deepest results in the 3-manifold topology in the last decade. We utilize the resolution of \emph{Geometrization Conjecture} by Perelman \cite{Perelman1}, \cite{Perelman2}, \cite{Perelman3}, which builds on the work of Hamilton \cite{Hamilton} and others. Moreover, we also utilize the resolution of \emph{Virtual Fibering Conjecture} for closed hyperbolic 3-manifolds by Agol \cite{Agol:VHC} and for non-positively curved manifolds by Przytycki and Wise \cite{Przytycki-Wise}.

We treat geometric 3-manifolds in \hyperref[sec:Geom3]{\S\ref*{sec:Geom3}}, aspherical 3-manifolds in \hyperref[sec:Asph3]{\S\ref*{sec:Asph3}}, and essential 3-manifolds in \hyperref[sec:Ess3]{\S\ref*{sec:Ess3}}. In many cases, \hyperref[thm:Closed-restated]{Theorem~\ref*{thm:Closed}} is sufficient to establish our isosystolic inequality, and this can be seen by direct topological arguments, possibly with a help of one of \emph{Virtual Fibering Theorems}. However, it should be noted that \hyperref[thm:Open-restated]{Theorem~\ref*{thm:Open}} turns out to be crucial for some closed graph manifolds, since we need to consider their infinite-degree \emph{non-compact} covers.

\subsection{Conventions}

We work in the category of smooth manifolds. As usual, a manifold is said to be \emph{closed} if it is compact and with empty boundary. When working with $n$-manifolds, the $n$-dimensional volume of $n$-chains and $n$-manifolds will be denoted by $\Vol$, and the $(n-1)$-dimensional volume of $(n-1)$-chains and hypersurfaces will be denoted by $\Area$. The open $r$-neighborhood and the closed $r$-neighborhood of a point $z$ in an $n$-manifold will be denoted by $B(z,r)$ and $\bar B(z,r)$ respectively; when $z$ is on an $(n-1)$-dimensional hypersurface $Z$ in an $n$-manifold $M$, the $(n-1)$-dimensional $r$-neighborhood on $Z$ with respect to the induced metric on $Z$ will be denoted by $D(z,r)$, while the notation $B(z,r)$ will be reserved for $n$-dimensional $r$-neighborhood in $M$. The length of curves and 1-chains will be denoted by $\Length$.

\subsection{Acknowledgement}

The author would like to thank Misha Katz for pointing out a mistake in the earlier draft, Yi Liu for instrumental comments on the recent results on non-positively curved 3-manifolds, and Joel Hass for helpful communications on minimal hypersurfaces. The author would also like to thank Larry Guth and Chris Croke for insightful discussions on systolic geometry.

\section{Closed Manifolds} \label{sec:Closed}

In this section, we prove \hyperref[thm:Closed]{Theorem~\ref*{thm:Closed}}, which establishes an isosystolic inequality for closed manifolds under a certain cohomological hypothesis. Following Guth's work \cite{Guth:Hypersurfaces}, we utilize nearly minimal hypersurfaces and show the existence of metric balls with large volume and small radius; as Guth noted, this approach can be loosely regarded as an adaptation of the proof of Geroch Conjecture for $n \leq 7$ by Schoen and Yau \cite{Schoen-Yau:Positive}. We then establish our isosystolic inequality with an observation, as in \cite{KKSSW}, that the volume estimate for these metric balls behaves well under covering maps. Throughout this section, we work strictly in the context of closed manifolds. We write $H_*(M;\bbZtwo)$ and $H^*(M;\bbZtwo)$ for singular homology and cohomology.

\subsection{Hypersurfaces} \label{ssec:Closed-Hypersurf}

Let $M$ be a closed smooth $n$-manifold. Recall first that, although a homology class cannot be represented by a smooth embedded submanifold in general, any codimension-one homology class $\bsZ \in H_{n-1}(M;\bbZtwo)$ can indeed be represented by a smooth embedded hypersurface \cite{Thom:VD}. This follows from Poincar\'e duality and the identification of $H^1(M;\bbZtwo)$ with the set of homotopy classes $[M,\RP^\infty]$ of maps from $M$ into $\RP^\infty \simeq K(\bbZtwo,1)$; see \hyperref[ssec:Open-Smooth]{\S\ref*{ssec:Open-Smooth}} for some detail.

Now, suppose that $(M,g)$ is a closed riemannian $n$-manifold. Given a homology class $\bsZ \in H_{n-1}(M;\bbZtwo)$, we would like to take a smooth embedded hypersurface representative $Z$ of $\bsZ$ whose geometry is well-controlled locally. Ideally, we would like to take $Z$ to be \emph{area-minimizing} in $\bsZ$, i.e. having the area that realizes the infimum of the area of all smooth embedded hypersurfaces in $\bsZ$; however, this is not possible in general. Nonetheless, we can take $Z$ to be \emph{nearly area-minimizing}, i.e. having the area that is arbitrarily close to the infimum. For $\delta \geq 0$, we say a smooth embedded hypersurface $Z$ is \emph{$\delta$-minimizing} in its homology class $\bsZ$ if $\Area(Z') \geq \Area(Z)-\delta$ for any other smooth embedded hypersurface $Z'$ representing $\bsZ$. By definition, for any $\delta>0$, we can always take a $\delta$-minimizing hypersurface $Z$ representing its homology class $\bsZ$.

\begin{rem}
For $\bbZ$-homology, the compactness theorem in geometric measure theory guarantees the existence of an area-minimizing \emph{integral current} in the homology class $\bsZ \in H_{n-1}(M;\bbZ)$ \cite{FF}, \cite[\S4.2, \S4.4]{Federer:GMT}; if $n \leq 7$, the regularity theorem then guarantees the area-minimizer to be a smooth embedded hypersurface \cite{Federer:Reg+Z2}, \cite[\S5.4]{Federer:GMT}. For $\bbZtwo$-homology, a version of the compactness theorem for \emph{flat $\bbZtwo$-chains} guarantees the existence of an area-minimizing flat $\bbZtwo$-chain in the homology class $\bsZ \in H_{n-1}(M;\bbZtwo)$ \cite{Ziemer}, \cite{Fleming:Zn}, \cite[\S4.2, \S4.4]{Federer:GMT}; if $n \leq 7$, a version of the regularity theorem, e.g. \cite{Federer:Reg+Z2} together with \cite[Lem.\,4.2]{Morgan:Z2}, then guarantees the area-minimizer to be a smooth embedded hypersurface.
\end{rem}

\subsection{Area Comparison} \label{ssec:Closed-ACL}

The key observation in \cite{Guth:Hypersurfaces} is that the area of a metric sphere can be controlled near $\delta$-minimizing hypersurfaces. Let us first recall Gromov's Curve Factoring Lemma \cite{Gromov:Metric} in a metric ball $\bar B(x,r)$. The original version by Gromov gives a factorization of a loop in $\bar B(x,r)$ as a product of short loops. The following version by Guth \cite{Guth:Hypersurfaces} gives an analogous factorization of a 1-cycle in $\bar B(x,r)$ as a sum of short 1-cycles.

\begin{lemalpha}[Curve Factoring Lemma] \label{lem:CFL}
Let $(M,g)$ be a complete riemannian $n$-manifold. Fix $r>0$ and $x \in M$, and let $\gamma$ be a 1-cycle in $\bar B(x,r)$. Then, for any $\varepsilon>0$, there exist 1-cycles $\gamma_1, \cdots, \gamma_k \subset \bar B(x,r)$ with $\Length(\gamma_i) \leq 2r+\varepsilon$ for all $i$, such that $\gamma$ is homologous to $\gamma_1 + \cdots + \gamma_k$.
\end{lemalpha}

Let $(M,g)$ be a closed riemannian $n$-manifold. For a non-zero homology class $\bsgamma \in H_1(M;\bbZtwo)$, we set $\Length(\bsgamma)$ to be the infimum of $\Length(\gamma)$ over all 1-cycles $\gamma$ representing $\bsgamma$. For a non-zero cohomology class $\bszeta \in H^1(M;\bbZtwo)$, we define
\[
\Length(\bszeta):=\inf \{ \ts \Length(\bsgamma) \mid \bsgamma \in H_1(M;\bbZtwo), \langle \bszeta,\bsgamma \rangle \neq 0 \ts \}.
\]
Equivalently, for any embedded hypersurface $Z$ representing the Poincar\'e dual $\bsZ \in H_{n-1}(M;\bbZtwo)$ of $\bszeta$, $\Length(\bszeta)$ is the infimum of $\Length(\gamma)$ over all 1-cycles $\gamma$ with a non-zero algebraic intersection number against $Z$. We have $\Length(\bszeta) \geq \Sys(M,g)>0$ for any non-zero class $\bszeta \in H^1(M;\bbZtwo)$.

We now give an estimate for the area of a metric sphere centered at a point on a $\delta$-minimizing hypersurface $Z$, in terms of the area of the disk on $Z$ with the same center and the radius. Our lemma improves a similar estimate for closed manifolds in the proof of Stability Lemma in \cite{Guth:Hypersurfaces}.

\begin{lem}[Area Comparison Lemma] \label{lem:ACL-C}
Let $(M,g)$ be a closed riemannian $n$-manifold. Let $\bszeta \in H^1(M;\bbZtwo)$ be a non-zero class with $\rho:=\Length(\bszeta)/2>0$, and let $\bsZ \in H_{n-1}(M;\bbZtwo)$ be the Poincar\'e dual of $\bszeta$. Let $\delta >0$, and let $Z \subset M$ be a smooth embedded hypersurface, $\delta$-minimizing in its homology class $\bsZ$. Then, for any $z \in Z$ and any $r \in (0,\rho)$,
\begin{align} \label{eqn:ACL-C}
\Area \big(\partial \bar B(z,r)\big) \geq 2 \, \Area \big(Z \cap \bar B(z, r) \big) -2\delta \geq 2 \, \Area \big(D(z, r) \big) -2\delta.
\end{align}
\end{lem}

\begin{rem}
The factor 2 in the inequality \hyperref[eqn:ACL-C]{(\ref*{eqn:ACL-C})} above is optimal for the class of manifolds in concern; see \hyperref[ssec:Ratio]{\S\ref*{ssec:Ratio}} for the related discussion.
\end{rem}

\begin{proof}
We first note that $Z \cap \bar B(z,r)$, regarded as a relative $(n-1)$-cycle in the pair $(\bar B(z, r), \partial \bar B(z,r))$, is null-homologous; we recall the argument from \cite{Guth:Hypersurfaces} for the convenience of readers. Suppose that is not the case. Then, by the Poincar\'e-Lefschetz duality, the cycle $Z \cap \bar B(z,r)$ has a non-zero algebraic intersection number with an absolute 1-cycle $\gamma$ in $\bar B(z, r)$. Since $r<\rho$, \hyperref[lem:CFL]{Lemma~\ref*{lem:CFL}} with small enough $\varepsilon >0$ yields a decomposition $\gamma=\sum_i\gamma_i$ with $\Length(\gamma_i) \leq 2r+\varepsilon < 2\rho = \Length(\bszeta)$ for all $i$. However, some short cycle $\gamma_i$ must have a non-zero intersection number against $Z$; this contradicts the definition of $\Length(\bszeta)$.

Now, the null-homologous relative $(n-1)$-cycle $Z \cap \bar B(z,r)$ bounds a relative $n$-chain $Q_1$ in $(\bar B(z,r),\partial \bar B(z,r))$. Since the cycle $Z \cap \bar B(z,r)$ is embedded and has multiplicity 1, we can take $Q_1$ to be a sum of the closure of some components of $\bar B(z,r) \ssminus Z$. Moreover, since each component of $Z \cap \bar B(z,r)$ is null-homologous, and hence 2-sided and separating, it follows that $Z \cap \bar B(z,r)$ also bounds the complementary $n$-chain $Q_0$ in $(\bar B(z, r), \partial \bar B(z,r))$. Hence, $\bar B(z,r)$ is decomposed along $Z \cap \bar B(z,r)$ into $Q_1$ and $Q_0$.

From the long exact sequence in homology, we see that the $(n-2)$-cycle $Z \cap \partial \bar B(z,r)$ in $\partial \bar B(z,r)$ is a boundary; indeed, this $(n-2)$-cycle bounds $(n-1)$-chains $P_1:=Q_1 \cap \partial \bar B(z,r)$ and $P_0:=Q_0 \cap \partial \bar B(z,r)$ in $\partial \bar B(z,r)$. Hence, $\partial \bar B(z,r)$ is decomposed along $Z \cap \partial \bar B(z,r)$ into $P_1$ and $P_0$. Among $P_1$ and $P_0$, let $P$ be the one whose area is smaller; then, we have
\begin{align} \label{eqn:Factor2-C}
\Area\big(\partial \bar B(z,r)\big) \geq 2 \, \Area(P).
\end{align}

Let $Z'$ be a $(n-1)$-cycle in $M$, formed by cutting out $Z \cap \bar B(z,r)$ from $Z$ and capping it off by $P$; we can modify the cycle $Z'$ by rounding off the corners and obtain a smooth hypersurface $Z''$ without increasing the area. Then, $Z$ is homologous to $Z'$, and hence to $Z''$. Since $Z$ is $\delta$-minimizing in its homology class, we have $\Area(Z') \geq \Area(Z'') \geq \Area(Z)-\delta$.
Since $Z'$ and $Z$ coincide in the complement of the closed ball $\bar B(z,r)$, we can subtract $\Area(Z \ssminus \bar B(z,r))$ from this inequality and obtain
\begin{align} \label{eqn:Minimizing-C}
\Area(P) \geq \Area \big(Z \cap \bar B(z,r)\big)-\delta.
\end{align}

The first inequality in \hyperref[eqn:ACL-C]{(\ref*{eqn:ACL-C})} now follows from \hyperref[eqn:Factor2-C]{(\ref*{eqn:Factor2-C})} and \hyperref[eqn:Minimizing-C]{(\ref*{eqn:Minimizing-C})}. The second inequality in \hyperref[eqn:ACL-C]{(\ref*{eqn:ACL-C})}, or equivalently $\Area(Z \cap \bar B(z,r)) \geq \Area (D(z,r))$, follows from $Z \cap \bar B(z,r) \supset D(z,r)$ and an observation that the distance measured in $M$ is at most the distance measured along $Z$.
\end{proof}

\subsection{Volume of Balls} \label{ssec:Closed-Balls}

We now give a volume estimate for balls with small diameters in a closed riemannian manifold $M$ whose $\bbZtwo$-cohomology has the maximal cup-length. The following estimate improves Guth's main theorem in \cite{Guth:Hypersurfaces} which finds a ball with volume at least $(R/4n)^n$ under the same hypothesis.

\begin{thm} \label{thm:Closed-Balls}
Let $(M,g)$ be a closed riemannian $n$-manifold. Suppose that there exist (not necessarily distinct) cohomology classes $\bszeta_1, \cdots, \bszeta_n$ in $H^1(M;\bbZtwo)$ such that $\bszeta_1 \cup \cdots \cup \bszeta_n \neq \bszero$ in $H^n(M;\bbZtwo)$ and $\rho:=\min_i \Length(\bszeta_i)/2>0$. Then, there exists $z \in M$ such that, for any $R \in (0,\rho)$,
\[
\Vol\big(B(z,R)\big) \geq \frac{(2R)^n}{n!}_.
\]
\end{thm}

\begin{proof}
We prove the statement by induction on the dimension $n$. The case $n = 1$ is trivial. As the induction step, for $n \geq 2$, let $\bszeta_1, \cdots, \bszeta_n \in H^1(M;\bbZtwo)$ with $\rho:=\min_i \Length(\bszeta_i)/2>0$ such that $\bszeta_1 \cup \cdots \cup \bszeta_n \neq \bszero$. Let $\delta>0$. Let $\bsZ \in H_{n-1}(M;\bbZtwo)$ be the Poincar\'e dual of $\bszeta_n$, and let $Z$ be a closed hypersurface, $\delta$-minimizing in its non-zero homology class $\bsZ$. We write $\bsM \in H_n(M;\bbZtwo)$ for the fundamental class for $M$. By Poincar\'e duality and the property of the cup product, we have
\[
\langle \bszeta_1 \cup \cdots \cup \bszeta_{n-1}, \bsZ \rangle=\langle \bszeta_1 \cup \cdots \cup \bszeta_{n-1}, \bsM \cap \bszeta_n \rangle=\langle \bszeta_1 \cup \cdots \cup \bszeta_n, \bsM \rangle \neq 0.
\]
Restricting cocycle representatives $\zeta_1, \cdots, \zeta_{n-1}$ of $\bszeta_1, \cdots, \bszeta_{n-1}$ to $Z$, we obtain cohomology classes $\bszeta'_1, \cdots, \bszeta'_{n-1} \in H^1(Z;\bbZtwo)$ such that $\bszeta'_1 \cup \cdots \cup \bszeta'_{n-1} \neq \bszero$ in $H^{n-1}(Z;\bbZtwo)$. Moreover, $\rho':=\min_i \Length(\bszeta'_i)/2 \geq \rho >0$, since $\Length(\bszeta'_i) \geq \Length(\bszeta_i)$ for $1 \leq i \leq n-1$. Hence, by the induction hypothesis, there is a point $z\in Z$ such that, for any $r \in (0,\rho')$,
\begin{align} \label{eqn:Closed-Codim1}
\Area \big(D(z,r)\big) \geq \frac{(2r)^{n-1}}{(n-1)!}_.
\end{align}
Now, let $R \in (0,\rho)$ and take a ball $B(z,R)$ in $M$, centered at $z \in Z$; note that we have $R < \rho'$ and $R < \Length(\bszeta_n)$. Then, together with \hyperref[lem:ACL-C]{Lemma~\ref*{lem:ACL-C}} and the estimate \hyperref[eqn:Closed-Codim1]{(\ref*{eqn:Closed-Codim1})} above, the coarea integration yields
\begin{align}
\begin{aligned} \label{eqn:Closed-Integral}
\Vol \big(B(z,R)\big)
&=\int_0^R \nts \Area \big(\partial \bar B(z,r)\big) \, dr
\geq \int_0^R \nts \left[ \ts 2\Area \big(D(z,r)\big) -2\delta \ts \right] dr\\
&\geq \int_0^R \left[ \ts 2 \ts \frac{(2r)^{n-1}}{(n-1)!}-2\delta \ts \right] dr
= \frac{(2R)^n}{n!}-2R\delta.
\end{aligned}
\end{align}
This estimate can be obtained for any $\delta>0$, but the center $z$ is on the $\delta$-minimizing hypersurface $Z$ which depends on $\delta>0$. Take a sequence of points $z_m$ such that the estimate \hyperref[eqn:Closed-Integral]{(\ref*{eqn:Closed-Integral})} holds for $\delta_m:=1/m$. Since $M$ is compact, we have a convergent subsequence $z_{m_i} \rightarrow z_\infty$. The inequality
\[
\Vol \big(B(z_\infty,R)\big) =\lim_{i \rightarrow \infty} \Vol \big(B(z_{m_i},R)\big) \geq \lim_{\begin{smallmatrix} i \rightarrow \infty \\ m_i \rightarrow 0 \end{smallmatrix}} \frac{(2R)^n}{n!}-\frac{2R}{m_i}=\frac{(2R)^n}{n!}_.
\]
holds, since $x \mapsto \Vol(B(x,R))$ is a continuous function on $M$.
\end{proof}

Given a covering map $p: (\hat M,\hat g) \rightarrow (M,g)$ between riemannian manifolds, metric balls with small enough radius maps isometrically between $\hat M$ and $M$ under a covering map. This observation was used in the work of Katz, Katz, Sabourau, Shnider, and Weinberger \cite{KKSSW} to control an isosystolic inequality of $M$ by that of $\hat M$. Utilizing this observation, we now establish our isosystolic inequality for closed manifolds.

\begin{thm} \label{thm:Closed}
Let $M$ be a closed smooth $n$-manifold, and let $\hat M$ be a finite-degree cover of $M$. Suppose that there exist (not necessarily distinct) cohomology classes $\bszeta_1, \cdots, \bszeta_n$ in $H^1(\hat M;\bbZtwo)$ such that $\bszeta_1 \cup \cdots \cup \bszeta_n \neq \bszero$ in $H^n(\hat M;\bbZtwo)$.
Then, for any riemannian metric $g$ on $M$,
\begin{align*}
\Sys(M,g)^n &\leq n! \, \Vol(M,g).
\end{align*}
\end{thm}

\begin{proof}
Since $M$ and its cover $\hat M$ are closed, we have $\min_i \Length(\bszeta_i) \geq \Sys(\hat M,\hat g) \geq \Sys(M,g)$. The cover $\hat M$ satisfies the hypothesis of \hyperref[thm:Closed-Balls]{Theorem~\ref*{thm:Closed-Balls}}. Hence, we can find a point $\hat z \in \hat M$ such that, for any $R \in (0,\Sys(M,g)/2)$,
\[
\Vol\big(B(\hat z,R)\big) \geq \frac{(2R)^n}{n!}_.
\]

Let $p:(\hat M, \hat g) \rightarrow (M,g)$ be the covering map,  and let $z=p(\hat z)$. We observe that $B(\hat z,R) \subset \hat M$ maps injectively onto $B(z,R) \subset M$. Otherwise, there exists a geodesic arc $\hat \gamma \subset B(\hat z,R)$ between two distinct preimages $\hat y', \hat y''$ in $B(\hat z,R)$ of a single point $y \in B(z,R)$, necessarily with $\Length(\hat \gamma)<2R$; then, $\hat \gamma$ projects onto a non-contractible  loop $\gamma \subset B(z,R) \subset M$ with $\Length(\gamma)<2R<\Sys(M,g)$, which is a contradiction. Now, since $B(\hat z,R)$ projects injectively onto $B(z,R)$, these balls are isometric. Hence, we have
\[
\frac{(2R)^n}{n!} \leq \Vol \big(B(\hat z,R)\big) = \Vol \big(B(z,R)\big) \leq \Vol(M,g).
\]
Taking the supremum over $R<\Sys(M,g)/2$ and rearranging the constant, we obtain the desired inequality $\Sys(M,g)^n \leq n! \, \Vol(M,g)$.
\end{proof}

\section{Non-Compact Manifolds} \label{sec:Open}

In this section, we prove \hyperref[thm:Open]{Theorem~\ref*{thm:Open}}, which establishes an isosystolic inequality for possibly non-compact manifolds under a cohomological hypothesis similar to the one in \hyperref[thm:Closed]{Theorem~\ref*{thm:Closed}}. Generally, ``ordinary'' homology and cohomology is often inadequate to study non-compact manifolds, due to the absence of Poincar\'e duality between them. In order to generalize the argument from \hyperref[sec:Closed]{\S\ref*{sec:Closed}} in the context of possibly non-compact manifolds, we employ homology with closed support and cohomology with compact support, in addition to ``ordinary'' homology with compact support and ``ordinary'' cohomology with closed support. Poincar\'e duality and cup-products in this general context serve as our essential tools.

\subsection{Homology and Cohomology} \label{ssec:Open-Homology}

Throughout this section, homology and cohomology will always mean \emph{singular} homology and cohomology. Generally, homology and cohomology can be defined with any paracompactifying family of support for chains and cochains. We note that ``ordinary'' homology $H_*(M;\bbA)$ and cohomology $H^*(M;\bbA)$ in common practice actually refer to homology \emph{with compact support} and cohomology \emph{with closed support}. Since we work with non-compact manifolds, homology \emph{with closed support} and cohomology \emph{with compact support} are instrumental. A presentation of cohomology with compact support can be found in introductory textbooks such as \cite{Hatcher:AT}; a sheaf-theoretic treatment of singular homology and cohomology with a general family of support can be found in \cite{Thom:Sq}, \cite{Bredon}.

Let $M$ be an $n$-manifold. We write $\bbA$ for a coefficient ring; we will later focus on $\bbA=\bbZtwo$. Following \cite{Thom:Sq}, we write $\Fsupp$ for the full family of closed support and $\Ksupp$ for the family of compact support, for chains and cochains in $M$. We write $H_*^\Fsupp(M;\bbA)$ for homology with closed support and $H_*^\Ksupp(M;\bbA)$ for homology with compact support. Similarly, we write $H^*_\Ksupp(M;\bbA)$ for cohomology with compact support and $H^*_\Fsupp(M;\bbA)$ for cohomology with closed support. For closed manifolds, these two families of support coincide, and hence we have $H_*^\Fsupp(M;\bbA)=H_*^\Ksupp(M;\bbA)=H_*(M;\bbA)$ and $H^*_\Ksupp(M;\bbA)=H^*_\Fsupp(M;\bbA)=H^*(M;\bbA)$. We note that, for manifolds with locally finite triangulations, homology with closed support agrees with locally finite simplicial homology.

For the families of support $\Fsupp$ and $\Ksupp$, cochains and chains of mixed support are paired algebraically; cochains with closed support are paired with chains with compact support as usual, and cochains with compact support are paired with chains with closed support. These pairings descend to homology and cohomology:
\begin{align*}
\langle \;\;,\:\: \rangle: H^k_\Fsupp(X;\bbA) \otimes H_k^\Ksupp(X;\bbA) & \rightarrow \bbA, \\
\langle \;\;,\:\: \rangle: H^k_\Ksupp(X;\bbA) \otimes H_k^\Fsupp(X;\bbA) & \rightarrow \bbA.
\end{align*}

The cup products are defined for cochains, with any combination of support; the support of the cup product is the intersection of the support of the factors. With families of support $\Fsupp$ and $\Ksupp$, we have three cup products:
\begin{align*}
\cup&: H^k_\Fsupp(X;\bbA) \otimes H^\ell_\Fsupp(X;\bbA) \rightarrow H^{k+\ell}_\Fsupp(X;\bbA), \\
\cup&: H^k_\Fsupp(X;\bbA) \otimes H^\ell_\Ksupp(X;\bbA) \rightarrow H^{k+\ell}_\Ksupp(X;\bbA), \\
\cup&: H^k_\Ksupp(X;\bbA) \otimes H^\ell_\Ksupp(X;\bbA) \rightarrow H^{k+\ell}_\Ksupp(X;\bbA).
\end{align*}
Similarly, the cap products are defined for chains and cochains, with any combination of support; as in the cup product, the support of the cap product is the intersection of the support of the factors. With families of support $\Fsupp$ and $\Ksupp$, we have four cap products:
\begin{align*}
\cap&: H_{k+\ell}^\Ksupp(M;\bbA) \otimes H^k_\Ksupp(M;\bbA) \rightarrow H_\ell^\Ksupp(M;\bbA), \\
\cap&: H_{k+\ell}^\Ksupp(M;\bbA) \otimes H^k_\Fsupp(M;\bbA) \rightarrow H_\ell^\Ksupp(M;\bbA), \\
\cap&: H_{k+\ell}^\Fsupp(M;\bbA) \otimes H^k_\Ksupp(M;\bbA) \rightarrow H_\ell^\Ksupp(M;\bbA), \\
\cap&: H_{k+\ell}^\Fsupp(M;\bbA) \otimes H^k_\Fsupp(M;\bbA) \rightarrow H_\ell^\Fsupp(M;\bbA).
\end{align*}

The primary advantage of utilizing homology with closed support to study non-compact $n$-manifolds is the existence of the fundamental class $\bsM=[M]$ in the top dimensional homology $H_n^\Fsupp(M;\bbZtwo)$, where the coefficient $\bbZtwo$ can be replaced with $\bbZ$ if $M$ is oreintable. Together with the third and the fourth cap products above, the fundamental class yields two versions of Poincar\'e duality, one for the family $\Ksupp$ of compact support and another for the family $\Fsupp$ of closed support: for any $n$-manifold $M$, we have
\begin{align*}
\PD & : H^k_\Ksupp(M;\bbZtwo) \stackrel{\cong}{\longrightarrow} H_{n-k}^\Ksupp(M;\bbZtwo),\\
\PD & : H^k_\Fsupp(M;\bbZtwo) \stackrel{\cong}{\longrightarrow} H_{n-k}^\Fsupp(M;\bbZtwo),
\end{align*}
where the coefficient $\bbZtwo$ can be replaced with $\bbZ$ if $M$ is orientable. Both duality maps are given by $\bszeta \mapsto \bsM \cap \bszeta$, with respective types of cap products.

\subsection{Smooth Hypersurfaces} \label{ssec:Open-Smooth}

Suppose that $M$ is a smooth $n$-manifold which is possibly non-compact. Just as for closed manifolds, it is not possible in general to represent a homology class, with compact support or with closed support, by a smooth embedded submanifold. However, it turns out that any codimension-one homology class, with compact support or with closed support, can be represented by a smooth embedded hypersurface.

Let us first consider a homology class $\bsZ \in H_{n-1}^\Fsupp(M;\bbZtwo)$ with closed support. Even if $M$ is non-compact, a smooth hypersurface representative $Z$ of $\bsZ$ exists by essentially the same argument as in the close manifold case \cite{Thom:VD}. We briefly review the argument. Let $\bszeta \in H^1_\Fsupp(M;\bbZtwo)$ be the Poincar\'e dual of $\bsZ \in H_{n-1}^\Fsupp(M;\bbZtwo)$. Through the bijection between $H^1_\Fsupp(M;\bbZtwo)$ and $[M,\RP^\infty]$ arising from $\RP^\infty \simeq K(\bbZtwo,1)$, the class $\bszeta$ corresponds to a homotopy class $[f] \in [M,\RP^\infty]$, so that $\bszeta$ is the pull-back $f^*(\bsw_1)$ of the unique non-zero class $\bsw_1 \in H^1_\Fsupp(\RP^\infty;\bbZtwo)$. By cellular approximation and Whitney smooth approximation, we may take the representative $f$ to be a composition $f=i \circ h$, where $h: M \rightarrow \RP^n$ is a smooth map and $i: \RP^n \hookrightarrow \RP^\infty$ is the standard inclusion. Since $\RP^{n-1} \subset \RP^n$ represents the Poincar\'e dual of $i^*(\bsw_1)$, the preimage of $\RP^{n-1} \subset \RP^n$ under $h$ represents the Poincar\'e dual of $\bszeta=f^*(\bsw_1)=h^*i^*(\bsw_1)$; hence, we may take $Z$ to be the preimage $h^{-1}(\RP^{n-1})$ of some copy of $\RP^{n-1} \subset \RP^n$, transverse to the map $h$. It is worthwhile to note that, if $M$ is non-compact, the hypersurface $Z$ may or may not be a closed hypersurface, depending on the class $\bsZ$.

Let us now consider a homology class $\bsZ \in H_{n-1}^\Ksupp(M;\bbZtwo)$ with compact support. For a non-compact manifold $M$, the argument above does not carry over directly to compactly supported cohomology classes, since we cannot make use of the Eilenberg-MacLane spectrum in the same manner; more specifically, $H^1_\Ksupp(M;\bbZtwo)$ is not in bijection with $[M,\RP^\infty]$. However, it is still possible to represent a non-zero homology class $\bsZ$ by a smooth embedded hypersurface $Z$.

\begin{lem} \label{lem:Codim1Rep}
For any smooth $n$-manifold $M$, a homology class $\bsZ \in H_{n-1}^\Ksupp(M;\bbZtwo)$ can be represented by a smooth embedded closed hypersurface $Z$.
\end{lem}

\begin{proof}
It suffices to consider a non-compact manifold $M$. Let $\bszeta \in H^1_\Ksupp(M;\bbZtwo)$ be the Poincar\'e dual of $\bsZ \in H_{n-1}^\Ksupp(M;\bbZtwo)$. Recall that we have $H^*_\Ksupp(M;\bbZtwo)=\varinjlim H^*_\Fsupp(M,M \ssminus K;\bbZtwo)$ where the direct limit is taken over the directed family of compact sets $K \subset M$; see, for example, \cite[\S3.3]{Hatcher:AT}. In particular, there exists a compact set $K$ and $\bszeta' \in H^1_\Fsupp(M,M \ssminus K;\bbZtwo)$ such that $\bszeta$ is the image of $\bszeta'$ in the directed limit; we may assume that $K$ is a compact $n$-submanifold of $M$ with non-empty boundary. There is also an isomorphism $H^1_\Fsupp(M,M \ssminus K;\bbZtwo) \cong H^1_\Fsupp(M/\overline{M \ssminus K};\bbZtwo)$ arising from excision; let $\bszeta_K \in H^1_\Fsupp(M/\overline{M \ssminus K};\bbZtwo)$ be the image of $\bszeta'$ under this isomorphism. 

Let $q: M \rightarrow M/\overline{M \ssminus K}$ be the quotient map. Now, $\bszeta_K$ corresponds to some homotopy class $[f_K] \in [M/\overline{M \ssminus K}, \RP^\infty]$, so that $\bszeta_K=f_K^*(\bsw_1)$ for the unique non-zero class $\bsw_1 \in H^1_\Fsupp(\RP^\infty;\bbZtwo)$. Modifying $f_K$ if necessary, we may assume that $f_K$ maps to $\RP^n \subset \RP^\infty$ and is smooth on $\interior K$; then, the preimage under $f_K$ of some copy  $\RP^{n-1} \subset f_K \circ q (\interior K)$ is a smooth embedded closed hypersurface $Z_K$ representing the Poincar\'e dual of $\bszeta_K$. It follows that the preimage $Z=q^{-1}(Z_K) \subset \interior K$ is a smooth embedded closed hypersurface representing the Poincar\'e-Lefschetz dual of $\bszeta'$, and hence representing the Poincar\'e dual $\bsZ$ of $\bszeta$.
\end{proof}

\begin{rem}
Due to the quotient map $q: M \rightarrow \overline{M \ssminus K}$, the representative $f:=f_K \circ q: M \rightarrow \RP^\infty$ is constant outside of a suitably chosen compact set $K$. The choice of the compact set $K$ may depend on the homology class $\bsZ$.
\end{rem}

\subsection{Nearly Minimizing Hypersurfaces} \label{ssec:Open-Minimizing}

Suppose now that $(M,g)$ is a complete riemannian $n$-manifold which is possibly non-compact. Given a codimension-one homology class, we shall consider a smooth hypersurface representative whose geometry is well-controlled locally, as in the compact case. Such hypersurfaces can be defined for the homology class with compact support or with closed support.

We first consider nearly area-minimizing hypersurface in a homology class $\bsZ \in H_{n-1}^\Ksupp(M;\bbZtwo)$ with compact support, which will be utilized in the proof of \hyperref[thm:Open]{Theorem~\ref*{thm:Open}}. A smooth hypersurface representative of $\bsZ$ is necessarily compact and hence with finite area. For $\delta \geq 0$, we say a smooth hypersurface $Z$ is \emph{$\delta$-minimizing} in its homology class $\bsZ$ if $\Area(Z') \geq \Area(Z)-\delta$ for any other smooth embedded hypersurface $Z'$ representing $\bsZ$. By definition, for any $\delta>0$, we can always take a $\delta$-minimizing hypersurface $Z$ representing its homology class $\bsZ$.

\begin{rem}
The compactness theorem again guarantees that an area-minimizing sequence of flat $\bbZtwo$-chains must have a subsequence whose limit is again a flat $\bbZtwo$-chain. Of course, if $n \geq 8$, the area-minimizer for the sequence of smooth hypersurfaces may fail to be a smooth hypersurface, just as in closed manifolds. It should also be pointed out that, if $M$ is a non-compact manifold, the area-minimizing limit of a sequence of flat chains in a homology class $\bsZ \in H_{n-1}^\Ksupp(M;\bbZtwo)$ may fail to remain in $\bsZ$; in particular, even in low dimensions where the regularity of the limit is guaranteed, an area-minimizing smooth hypersurface may fail to exists in a homology class $\bsZ$. 

We note two situations where non-homologous limit can be witnessed. First, an area-minimizer may be non-compact even if the sequence consists of compactly supported chains; see \cite{Morrey:Plateau}, \cite[Ex.\,6.2]{Hass-Scott:LeastArea}. In such situations, the limit is not compactly supported, and hence it does not belong to the compactly supported homology class $\bsZ$. Second, we could have an area-minimizing sequence with some components escaping into the ends of the manifold. For example, in a non-compact finite-volume hyperbolic manifold, we can take a sequence of frontiers of a cusp, escaping into the cusp with the area decreasing to zero. One can also modify the metric on $M$ in this example so that the area of the frontiers decrease towards some positive infimum; see \cite[Ex.\,6.1]{Hass-Scott:LeastArea}. If $M$ has at least two ends, these frontiers represent a non-zero homology class in $H_{n-1}^\Ksupp(M;\bbZtwo)$. However, the limit of these frontiers escaping into a cusp is the empty $\bbZtwo$-chain, and hence is null-homologous.

It should be emphasized that none of these delicate matters becomes an issue in our work, since we will work with \emph{compact} hypersurfaces that are nearly area-minimizing hypersurfaces in its homology class.
\end{rem}

Although it is not necessary in the proof of \hyperref[thm:Open]{Theorem~\ref*{thm:Open}}, let us briefly discuss a nearly area-minimizing hypersurface in a homology class $\bsZ \in H_{n-1}^\Fsupp(M;\bbZtwo)$ with closed support. A smooth hypersurface representative of $\bsZ$ may be non-compact, and hence it is possible that every representative has infinite area. We shall control the geometry of such hypersurfaces by a local condition. For $\rho>0$ and $\delta>0$, we say $Z$ is \emph{$(\rho,\delta)$-minimizing} in its homology class $\bsZ$ if the following condition is met: for any $x \in M$, if $Z'$ is another smooth embedded hypersurface representing the same class $\bsZ$ such that $Z$ and $Z'$ agree outside of $B(x,\rho)$, then $\Area(Z' \cap B(x,\rho)) \geq \Area(Z \cap B(x,\rho))-\delta$. For any $\rho>0$ and $\delta>0$, we can take a $(\rho,\delta)$-minimizing hypersurface in its homology class $\bsZ$.

\begin{rem}
In a non-compact complete manifold $M$, a smooth embedded hypersurface is said to be \emph{locally area-minimizing} in its homology class if it is a $(\infty,0)$-minimizing hypersurface, i.e. a hypersurface that are $(\rho,0)$-minimizing for any $\rho>0$. A version of the compactness theorem for non-compact flat $\bbZtwo$-chains is known \cite[\S27.3, \S31.2]{Simon:GMT}, and it guarantees that a locally area-minimizing sequence of flat $\bbZtwo$-chains must have a subsequence whose limit is again a flat $\bbZtwo$-chain. Of course, for the sequence of smooth hypersurfaces in a homology class $\bsZ \in H_{n-1}^\Fsupp(M;\bbZtwo)$, the area-minimizer may fail to be a smooth hypersurface if $n \geq 8$. However, this appears to be the only obstruction to the existence of smooth area-minimizing hypersufaces; in other words, it may be the case that the area-minimizer always remain in $\bsZ$. We are not sure if this aspect is thoroughly documented in the literature, although some results may be known, possibly as a folklore.
\end{rem}

\subsection{Area Comparison} \label{ssec:Open-ACL}

Let $(M,g)$ be a complete riemannian $n$-manifold. We give versions of \hyperref[lem:ACL-C]{Lemma~\ref*{lem:ACL-C}} for non-compact manifolds, in terms of homology with compact support and homology with closed support. The compactly supported case, in particular, will play an essential role in deriving our isosystolic inequality.

Let us first assure the validity of the version of \hyperref[lem:ACL-C]{Lemma~\ref*{lem:ACL-C}} for homology and cohomology with compact support, which will be utilized in the proof of \hyperref[thm:Open]{Theorem~\ref*{thm:Open}}. For a non-zero homology class $\bszeta \in H^1_\Ksupp(M;\bbZtwo)$, we define
\[
\Length(\bszeta):=\inf \{ \ts \Length(\bsgamma) \mid \bsgamma \in H_1^\Fsupp(M;\bbZtwo), \langle \bszeta,\bsgamma \rangle \neq 0 \ts \}
\]
where $\Length(\bsgamma)$ for $\bsgamma \in H_1^\Fsupp(M;\bbZtwo)$ is defined to be the infimum of $\Length(\gamma)$ over all (possibly non-compact) 1-cycles $\gamma$ representing $\bsgamma$. It should be understood that $\Length(\bsgamma)=\infty$ if $\bsgamma$ can only be represented by a non-compact cycle; hence, it is possible to have $\Length(\bszeta)=\infty$ if $M$ is non-compact. We note that $\Length(\bszeta) \geq \Sys(M,g)$ for any non-zero class $\bszeta \in H^1_\Ksupp(M;\bbZtwo)$, but it is possible to have $\Length(\bszeta)=\Sys(M,g)=0$ if $M$ is non-compact; in the following lemma, we require that $\Length(\bszeta)>0$.

\begin{lem}[Area Comparison Lemma, supported in $\Ksupp$] \label{lem:ACL-K}
Let $(M,g)$ be a complete riemannian $n$-manifold. Let $\bszeta \in H^1_\Ksupp(M;\bbZtwo)$ be a non-zero class with $\rho:=\Length(\bszeta)/2>0$, and let $\bsZ \in H_{n-1}^\Ksupp(M;\bbZtwo)$ be the Poincar\'e dual of $\bszeta$. Let $\delta >0$, and let $Z \subset M$ be a smooth embedded hypersurface, $\delta$-minimizing in its homology class $\bsZ$. Then, for any $z \in Z$ and any $r \in (0,\rho)$,
\begin{align*}
\Area \big(\partial \bar B(z,r)\big) \geq 2 \, \Area \big(Z \cap \bar B(z, r) \big) -2\delta \geq 2 \, \Area \big(D(z, r) \big) -2\delta.
\end{align*}
\end{lem}

\begin{proof}
The proof is identical to that of \hyperref[lem:ACL-C]{Lemma~\ref*{lem:ACL-C}}.
\end{proof}

Although it is not necessary in the proof of \hyperref[thm:Open]{Theorem~\ref*{thm:Open}}, we also record the version of \hyperref[lem:ACL-C]{Lemma~\ref*{lem:ACL-C}} for homology and cohomology with closed support. For a non-zero homology class $\bszeta \in H^1_\Fsupp(M;\bbZtwo)$, we define
\[
\Length(\bszeta):=\inf \{ \ts \Length(\bsgamma) \mid \bsgamma \in H_1^\Ksupp(M;\bbZtwo), \langle \bszeta,\bsgamma \rangle \neq 0 \ts \}
\]
where $\Length(\bsgamma)$ for $\bsgamma \in H_1^\Ksupp(M;\bbZtwo)$ is defined to be the infimum of $\Length(\gamma)$ over all (compact) 1-cycles $\gamma$ representing $\bsgamma$. Since we pair $\bszeta$ with compactly supported homology classes, $\Length(\bszeta)$ is always finite. We note that, as in the compactly supported case, we have $\Length(\bszeta) \geq \Sys(M,g)$ for any non-zero class $\bszeta \in H^1_\Fsupp(M;\bbZtwo)$, but it is possible to have $\Length(\bszeta)=\Sys(M,g)=0$ if $M$ is non-compact; in the following lemma, we require $\Length(\bszeta)>0$.

\begin{lem}[Area Comparison Lemma, supported in $\Fsupp$] \label{lem:AFC-F}
Let $(M,g)$ be a complete riemannian $n$-manifold. Let $\bszeta \in H^1_\Fsupp(M;\bbZtwo)$ be a non-zero class with $\rho:=\Length(\bszeta)/2>0$, and let $\bsZ \in H_{n-1}^\Fsupp(M;\bbZtwo)$ be the Poincar\'e dual of $\bszeta$. Let $\delta >0$, and let $Z \subset M$ be a smooth embedded hypersurface, $(\rho,\delta)$-minimizing in its homology class $\bsZ$. Then, for any $z \in Z$ and any $r \in (0,\rho)$,
\begin{align*}
\Area \big(\partial \bar B(z,r)\big) \geq 2 \, \Area \big(Z \cap \bar B(z, r) \big) -2\delta \geq 2 \, \Area \big(D(z, r) \big) -2\delta.
\end{align*}
\end{lem}

\begin{proof}
The proof is essentially identical to that of \hyperref[lem:ACL-C]{Lemma~\ref*{lem:ACL-C}}; we must make sure that the modification and smoothing of the hypersurface $Z$ take place locally within a ball $B(z,\rho) \supset B(z,r)$. 
\end{proof}

\subsection{Volume of Balls}

We now give a volume estimate for balls with small diameters in a complete riemannian manifold $M$. As in \hyperref[thm:Closed-Balls]{Theorem~\ref*{thm:Closed-Balls}}, we work under a hypothesis that $\bbZtwo$-cohomology has a maximal-cuplength, i.e. having first cohomology classes $\bszeta_1, \cdots, \bszeta_n$ with a non-vanishing cup-product. We will allow these classes to have mixed support so that our estimate applies to a broad class of manifolds; however, we must require that at least one of them has a compact support, so that the cup-product ends up in $H^n_\Ksupp(M;\bbZtwo) \cong \bbZtwo$, and not in $H^n_\Fsupp(M;\bbZtwo) \cong 0$.

\begin{thm} \label{thm:Open-Balls}
Let $(M,g)$ be a complete riemannian $n$-manifold. Suppose that there exist (not necessarily distinct) cohomology classes $\bszeta_1, \cdots, \bszeta_n$ in $H^1_\Fsupp(M;\bbZtwo)$ or $H^1_\Ksupp(M;\bbZtwo)$, with at least one of them in $H^1_\Ksupp(M;\bbZtwo)$, such that $\bszeta_1 \cup \cdots \cup \ts \bszeta_n \neq \bszero$ in $H^n_\Ksupp(M;\bbZtwo)$ and $\rho:=\min_i \Length(\bszeta_i)/2>0$. Then, for any $\delta>0$, there exists $z \in M$ such that, for any $R \in (0,\rho)$,
\[
\Vol\big(B(z,R)\big) \geq \frac{(2R)^n}{n!}-2R\delta.
\]
\end{thm}

\begin{rem}
In \hyperref[thm:Open-Balls]{Theorem~\ref*{thm:Open-Balls}}, we require a condition $\rho:=\min_i \Length(\bszeta_i)/2>0$ in order to utilize \hyperref[lem:ACL-K]{Lemma~\ref*{lem:ACL-K}}. This is a non-trivial condition if $M$ is non-compact.
\end{rem}

\begin{proof}
We may assume $n \geq 2$, since the statement holds trivially for $n=1$. We may also assume, without loss of generality, that $\bszeta_n \in H^1_\Ksupp(M;\bbZtwo)$.

Let $\delta>0$. Let $\bsZ \in H_{n-1}^\Ksupp(M;\bbZtwo)$ be the Poincar\'e dual of $\bszeta_n$, and let $Z$ be a closed hypersurface, $\delta$-minimizing in its non-zero homology class $\bsZ$. We write $\bsM \in H_n^\Fsupp(M;\bbZtwo)$ for the fundamental class for $M$. As in the proof of \hyperref[thm:Closed-Balls]{Theorem~\ref*{thm:Closed-Balls}}, by Poincar\'e duality and the property of the cup product, we have
\[
\langle \bszeta_1 \cup \cdots \cup \bszeta_{n-1}, \bsZ \rangle=\langle \bszeta_1 \cup \cdots \cup \bszeta_{n-1}, \bsM \cap \bszeta_n \rangle=\langle \bszeta_1 \cup \cdots \cup \bszeta_n, \bsM \rangle \neq 0.
\]
Restricting cocycle representatives $\zeta_1, \cdots, \zeta_{n-1}$ of $\bszeta_1, \cdots, \bszeta_{n-1}$ to $Z$, we obtain \emph{compactly supported} cohomology classes $\bszeta'_1, \cdots, \bszeta'_{n-1} \in H^1_\Ksupp(Z;\bbZtwo)=H^1(Z;\bbZtwo)$ such that $\bszeta'_1 \cup \cdots \cup \bszeta'_{n-1} \neq \bszero$ in $H^{n-1}_\Ksupp(Z;\bbZtwo)=H^{n-1}(Z;\bbZtwo)$. Moreover, $\rho':=\min_i \Length(\bszeta'_i)/2 \geq \rho >0$, since $\Length(\bszeta'_i) \geq \Length(\bszeta_i)$ for $1 \leq i \leq n-1$. Hence, by \hyperref[thm:Closed-Balls]{Theorem~\ref*{thm:Closed-Balls}}, there is a point $z\in Z$ such that, for any $r \in (0,\rho')$,
\begin{align} \label{eqn:Open-Codim1}
\Area \big(D(z,r)\big) \geq \frac{(2r)^{n-1}}{(n-1)!}_.
\end{align}
Now, let $R \in (0,\rho)$ and take a ball $B(z,R)$ in $M$ centered at $z \in Z$ above; note that we have $R < \rho'$ and $R < \Length(\bszeta_n)$. So, together with \hyperref[lem:ACL-K]{Lemma~\ref*{lem:ACL-K}} and the estimate \hyperref[eqn:Open-Codim1]{(\ref*{eqn:Open-Codim1})} above, the coarea integration identical to \hyperref[eqn:Closed-Integral]{(\ref*{eqn:Closed-Integral})} yields the desired inequality.
\end{proof}

Unlike \hyperref[thm:Closed-Balls]{Theorem~\ref*{thm:Closed-Balls}}, we may not be able to find a ball $B(z,R)$ with the volume bound $(2R)^n/n!$ without the error term; there is no obvious argument that guarantees that the balls $B(z,R)$ remain in a compact subset of $M$ as we vary $\delta \searrow 0$. Nonetheless, we can obtain the estimate of the volume of the ambient manifold by passing to the limit as $\delta \searrow 0$ and deduce an isosystolic inequality.

\begin{thm} \label{thm:Open}
Let $M$ be a (not necessarily closed) smooth $n$-manifold, and let $\hat M$ be a (possibly infinite-degree) cover of $M$. Suppose that there exist (not necessarily distinct) cohomology classes $\bszeta_1, \cdots, \bszeta_n$ in $H^1_\Fsupp(M;\bbZtwo)$ or $H^1_\Ksupp(M;\bbZtwo)$, with at least one of them in $H^1_\Ksupp(M;\bbZtwo)$, such that $\bszeta_1 \cup \cdots \cup \ts \bszeta_n \neq \bszero$ in $H^n_\Ksupp(M;\bbZtwo)$. Then, for any complete riemannian metric $g$ on $M$,
\begin{align*}
\Sys(M,g)^n &\leq n! \, \Vol(M,g).
\end{align*}
\end{thm}

\begin{proof}
If $\Sys(M,g)=0$, an isosystolic inequality holds for $(M,g)$ with any isosystolic constant. So, we may assume $\Sys(M,g)>0$. The cover $\hat M$ satisfies the hypothesis of \hyperref[thm:Open-Balls]{Theorem~\ref*{thm:Open-Balls}}, and $\min_i \Length(\bszeta_i) \geq \Sys(\hat M,\hat g) \geq \Sys(M,g)$. So, for each $\delta>0$, we can find a point $\hat z \in \hat M$ such that, for any $R \in (0,\Sys(M,g)/2)$,
\[
\Vol\big(B(\hat z,R)\big) \geq \frac{(2R)^n}{n!}-2R\delta.
\]
Now, recall that the covering argument used in the proof of \hyperref[thm:Closed-Balls]{Theorem~\ref*{thm:Closed-Balls}} does not require the degree of the covering map $p: \hat M \rightarrow M$ to be finite. So, as before, we can project the ball $B(\hat z, R)$ isometrically onto $B(z,R) \subset M$. Hence, we have
\[
\frac{(2R)^n}{n!}-2R\delta \leq \Vol \big(B(\hat z,R)\big) = \Vol \big(B(z,R)\big) \leq \Vol(M,g).
\]
Since this inequality holds for any $\delta>0$, we have $(2R)^n/n! \leq \Vol(M,g)$. Taking the supremum over $R<\Sys(M,g)/2$ and rearranging the constant, we obtain the desired inequality $\Sys(M,g)^n \leq n! \, \Vol(M,g)$.
\end{proof}

\begin{rem}
In addition to \hyperref[thm:Gromov]{Theorem~\ref*{thm:Gromov}}, Gromov established a universal isosystolic inequality for non-compact $n$-manifolds that are \emph{essential relative to infinity} \cite[\S4.4-4.5]{Gromov:Filling}; here, a non-compact $n$-manifold $M$ is said to be \emph{essential relative to infinity} if there is an aspherical space $K$ and a map $f:M \rightarrow K$, which is constant outside some compact subset in the interior of $M$, such that $f$ represents a non-zero class in $H_n(K;\bbA)$, with $\bbA=\bbZ$ if $M$ is orientable and $\bbA=\bbZtwo$ if $M$ is non-orientable. We remark that, although it is not immediate from the statement, the manifolds in \hyperref[thm:Open]{Theorem~\ref{thm:Open}} are probably essential relative to infinity.
\end{rem}

We are often primarily interested in isosystolic inequalities for closed manifolds. We point out that \hyperref[thm:Open]{Theorem~\ref*{thm:Open}} can play an essential role in establishing an isosystolic inequality for closed manifolds; even if a closed manifold and its finite-degree cover do not satisfy the hypothesis of \hyperref[thm:Closed]{Theorem~\ref*{thm:Closed}}, the same closed manifold and its \emph{infinite-degree} non-compact cover may satisfy the hypothesis of \hyperref[thm:Open]{Theorem~\ref*{thm:Open}}. This observation is crucial in \hyperref[sec:Asph3]{\S\ref*{sec:Asph3}}, where an isosystolic inequality is established for all closed aspherical 3-manifolds. We remark that an infinite-degree \emph{abelian} cover of a closed manifold has been utilized in the study of isosystolic inequalities in \cite{Katz-Lescop}.

\section{Tori and Real Projective Spaces} \label{sec:TRP}

The primary examples of manifolds whose $\bbZtwo$-cohomology have the maximal cup-length are tori $T^n$ and real projective spaces $\RP^n$ in all dimensions; the following inequality follows immediately from \hyperref[thm:Closed]{Theorem~\ref*{thm:Closed}}.

\begin{coro} \label{coro:TRP}
For any riemannian metric $g$ on $T^n$ and $\RP^n\nts$,
\begin{align*}
\Sys(T^n\nts,g)^n &\leq n! \, \Vol(T^n\nts,g) \\
\Sys(\RP^n\nts,g)^n &\leq n! \, \Vol(\RP^n\nts,g).
\end{align*}
\end{coro}

For $n=2$, the optimal inequalities are given by Loewner's inequality (\hyperref[thm:Loewner]{Theorem~\ref*{thm:Loewner}}) for $T^2$ and Pu's inequality (\hyperref[thm:Pu]{Theorem~\ref*{thm:Pu}}) for $\RP^2$ in \cite{Pu}. For $n \geq 3$, our isosystolic inequalities for $T^n$ and $\RP^n$ improve upon the best previously known inequalities by Guth (\hyperref[thm:Guth]{Theorem~\ref*{thm:Guth}}) in \cite{Guth:Hypersurfaces}.

A better understanding of systoles in $T^n$ and $\RP^n$ is highly desirable, as they form two of the most fundamental families of closed manifolds. We collect a few observations on the isosystolic inequalities for $T^n$ and $\RP^n\nts$, and related manifolds.

\subsection{Compact Space Forms} \label{ssec:CSF}

Recall that \hyperref[thm:Closed]{Theorem~\ref*{thm:Closed}} allows us to obtain an isosystolic inequality for a manifold $M$ by lifting a large metric ball in $M$ to its cover $\hat M$. We observe that, by taking $\hat M$ to be the torus $T^n$ or the real projective space $\RP^n\nts$, we obtain isosystolic inequalities for some compact space forms.

\begin{thm} \label{thm:CSF}
Let $M$ be a closed $n$-manifold. Suppose that $M$ is homeomorphic~to either (i) a compact euclidean space forms, such as the $n$-torus $T^n\nts$, or (ii) a spherical space form with even-order fundamental group, such as the real projective $n$-space $\RP^n\nts$. Then, for any riemannian metric $g$ on $M$,
\begin{align} \label{eqn:CSF}
\Sys(M,g)^n &\leq n! \, \Vol(M,g).
\end{align}
\end{thm}

\subsubsection*{Euclidean Space Forms}

A compact euclidean space form is a compact quotient of the euclidean space $\Euc^n=\left(\bbsR^n\nts,g_\euc\right)$, i.e. $\bbsR^n$ equipped with the flat metric $g_\euc$, by a discrete group $\varGamma$ of isometries acting freely on $\Euc^n\nts$.

\begin{proof}[Proof of {\hyperref[thm:CSF]{Theorem~\ref*{thm:CSF}}} for euclidean space forms]
Suppose that $M$ is homeomorphic to a compact space form $\Euc^n\nts/\varGamma$. By the classical work of Bieberbach, it is well-known that $\varGamma$ contains a finite-index normal subgroup $\varLambda \cong \bbZ^n\nts$. Hence, $\Euc^n\nts/\varGamma$ admits a finite-degree covering by a flat $n$-torus $\Euc^n\nts/\varLambda$, and $M$ admits a finite-degree covering by $\hat M$ homeomorphic to $T^n\nts$. The inequality \hyperref[eqn:CSF]{(\ref*{eqn:CSF})} now follows from \hyperref[thm:Closed]{Theorem~\ref{thm:Closed}}.
\end{proof}

\subsubsection*{Spherical Space Forms}

A spherical space form is a quotient of the standard sphere $\Sph^n=\left(S^n\nts,g_\sph\right)$, i.e. $S^n$ equipped with the standard metric $g_\sph$ with constant sectional curvature $\kappa \equiv 1$, by a discrete group $\varGamma$ of isometries acting freely on $\Sph^n\nts$. By elementary linear algebra, one can verify  that (i) when $n$ is odd, $\Sph^n\nts/\varGamma$ must be orientable, and (ii) when $n$ is even, $\Sph^n\nts/\varGamma$ must be $\RP^n\nts$, which is necessarily non-orientable, or $\Sph^n$ itself.

\begin{proof}[Proof of {\hyperref[thm:CSF]{Theorem~\ref*{thm:CSF}}} for spherical space forms]
Suppose that $M$ is homeomorphic to a compact space form $\Sph^n\nts/\varGamma$. If the order of $\varGamma$ is even, $\varGamma$ contains a subgroup $\varLambda$ of order 2 which is generated necessarily by the antipodal map. Hence, $\Sph^n\nts/\varGamma$ admits a finite-degree covering by $\RP^n=\Sph^n\nts/\varLambda$, and $M$ admits a finite-degree covering by $\hat M$ homeomorphic to $\RP^n\nts$. The inequality \hyperref[eqn:CSF]{(\ref*{eqn:CSF})} now follows from \hyperref[thm:Closed]{Theorem~\ref{thm:Closed}}.
\end{proof}

\subsection{Exotic Space Forms} \label{ssec:Exotic}

It is worthwhile to emphasize that the hypotheses in Gromov's isosystolic inequality (\hyperref[thm:Gromov]{Theorem~\ref*{thm:Gromov}}), Guth's inequality (\hyperref[thm:Guth]{Theorem~\ref*{thm:Guth}}) and our inequality in \hyperref[thm:Closed]{Theorem~\ref*{thm:Closed}} are topological, and they only concern the homotopy type of the manifolds. Hence, we can obtain isosystolic inequality for a manifold without knowing anything about the metrics that can be put on or the smooth structure of the manifold.

Under some hypothesis, it may happen that homotopy equivalence $M \simeq M'$ implies the existence of a diffeomorphism $M \cong M'$; however, in general, there are manifolds that are homotopy equivalent but not diffeomorphic. Whether there exists a manifold homotopy equivalent to, but not diffeomorphic to, a given manifold $M$ is a difficult question on its own. Let us collect a few examples of \emph{exotic} $\RP^n$ that are homotopy equivalent to, but not diffeomorphic to, the standard $\RP^n\nts$.

\subsubsection*{Quotients of Exotic $S^n$}

For $n \geq 7$, there are many exotic $n$-spheres, i.e. manifolds homeomorphic to $S^n$ but not diffeomorphic to the standard $S^n$. The first example is due to Milnor \cite{Milnor:Exotic} who constructed an exotic 7-sphere $M^7\nts$. It is an $S^3$-bundle over $S^4\nts$, and the antipodal map on fibers defines a smooth involution $\tau$ acting freely on $M^7\nts$. The smooth quotient $Q^7:=M^7\nts/\tau$ is homotopy equivalent to, but not diffeomorphic or PL-homeomorphic to, the standard $\RP^7$ \cite{Hirsch-Milnor}.

\subsubsection*{Exotic Quotients of $S^n$}

For $n \geq 5$, there are also many exotic smooth involutions acting freely on the standard $S^n\nts$. For example, the smooth involution $\tau$ on the Milnor manifold $M^7\nts$, in the discussion above, fixes smoothly embedded $S^6$ and $S^5$ with the standard smooth structures, but the restriction of $\tau$ to these spheres are exotic. The smooth quotients $Q^6:=S^6\nts/\tau$ and $Q^5:=S^5\nts/\tau$ are homotopy equivalent to, but not diffeomorphic to, the standard $\RP^6$ and $\RP^5$ respectively \cite{Hirsch-Milnor}.

\subsubsection*{Quotients of Potentially Exotic $S^n$}

Cappell and Shaneson \cite{Cappell-Shaneson} obtained an intriguing family of manifolds $Q^4$ from $\RP^4$ by removing the complement of a regular neighborhood of the standard $\RP^2 \subset \RP^4$ and replacing it with a particular bundle over $S^1$ with punctured $T^3$ fibers. Each $Q^4$ arising from their construction is homotopy equivalent to, but not diffeomorphic to, the standard $\RP^4$; its double-cover $M^4$ is homotopy equivalent to $S^4\nts$, and hence homeormophic to $S^4$ by the work of Freedman \cite{Freedman}. Some of those Cappell-Shaneson manifolds $M^4$ are known to be diffeomorphic to the standard $S^4\nts$, cf. \cite{Gompf:Killing}, \cite{Akbulut}, \cite{Gompf:More}; however it is not known if all of them are, and they comprise an important potential counterexamples to smooth Poincar\'{e} Conjecture in dimension 4.

\subsubsection*{Inequalities for Exotic Space Forms} 

We record a generalization of \hyperref[thm:CSF]{Theorem~\ref*{thm:CSF}}, allowing the manifold $M$ to be homotopy equivalent to compact space forms.

\begin{coro} \label{coro:Exotic}
Let $M$ be a closed $n$-manifold. Suppose that $M$ is homotopy equivalent to either (i) a compact euclidean space forms, such as $T^n\nts$, or (ii) a spherical space form with even-order fundamental group, such as $\RP^n\nts$. Then, for any riemannian metric $g$ on $M$,
\begin{align*}
\Sys(M,g)^n &\leq n! \, \Vol(M,g).
\end{align*}
\end{coro}

Although this generalization is based on a trivial observation, it has somewhat non-trivial consequences. We emphasize that the inequality requires no knowledge of the smooth structure on the manifold $M$, let alone any metrics on it. For instance, regardless of the smooth structure, every Cappell-Shaneson manifolds $Q^4 \simeq \RP^4$ satisfies the isosystolic inequality $\Sys(Q,g)^4 \leq 24 \, \Vol(Q,g)$.

\subsubsection*{Remark on Babenko's Work}

If $n$ is odd so that $\RP^n$ is orientable, it has been known by the work of Babenko in \cite{Babenko:Asymptotic} and \cite{Babenko:Z2} that the standard $\RP^n$ and any homotopy $\RP^n$ share the same optimal isosystolic constant; hence, an isosystolic inequality holds for the standard odd-dimensional $\RP^n$ if and only if the inequality with the same isosystolic constant holds for an exotic $\RP^n\nts$, regardless of how we obtained our inequality for the standard $\RP^n$ in the first place. It is not known to the present author if the analogous statement holds in even dimensions, where $\RP^n$ is non-orientable. In even dimensions, though Babenko's result does not apply, we were still able to establish the inequality in \hyperref[coro:Exotic]{Corollary~\ref*{coro:Exotic}}, since every homotopy $\RP^n$ satisfies the cohomological criterion of \hyperref[thm:Closed]{Theorem~\ref*{thm:Closed}}.

\subsection{Sphere-Disk Area Ratio} \label{ssec:Ratio}

In \hyperref[coro:TRP]{Corollary~\ref*{coro:TRP}}, we saw that isosystolic inequality holds with the constant $C_n=n!$ for tori $T^n$ and real projective spaces $\RP^n\nts$. By the Stirling's approximation, the factorial $C_n=n!$ is approximated by $\sqrt{2 \pi n} (n/e)^n\nts$. The growth rate of the optimal constants $C_n$ is expected to be slower, in the order of $n^{n/2}$, both for $T^n$ and $\RP^n\nts$. Here, we shall discuss a few ideas on how one may, and may not, improve the isosystolic constants for $T^n$ and $\RP^n\nts$.

\subsubsection*{Unwinding the Inequality}

Recall that our isosystolic inequality comes from the volume estimate of a ball in \hyperref[thm:Closed-Balls]{Theorem~\ref*{thm:Closed-Balls}}. If we unwind the inductive proof, we find that our volume estimate can be written as
\begin{align} \label{eqn:OriginalFactors}
\Vol(B^n\ntts) \geq \bigg(\frac{R}{n} \cdot 2\bigg)\bigg(\frac{R}{n-1} \cdot 2\bigg) \cdots \bigg(\frac{R}{1} \cdot 2\bigg)= \frac{(2R)^n}{n!}_,
\end{align}
where $B^n=B^n(z,R)$ is a ball of radius $R \in(0,\Sys(M,g)/2)$ with a suitably chosen center $z$. In each step of the induction, going from the dimension $k$ to $k+1$, we pick up two factors, $2$ and $R/(k+1)$; the factor $2$ comes from the area comparison estimate \hyperref[eqn:ACL-C]{(\ref*{eqn:ACL-C})} in \hyperref[lem:ACL-C]{Lemma~\ref{lem:ACL-C}}, and the factor $R/(k+1)$ comes out of integrating the area of $k$-spheres in $(k+1)$-submanifolds with the coarea formula.

Suppose that we seek for a better inequality by refining this inductive argument, following the same outline of the argument. We shall focus on the factor $2$ from \hyperref[lem:ACL-C]{Lemma~\ref{lem:ACL-C}}, which represents the approximate \emph{sphere-disk area ratio bound}
\[
\Area\big(S^k(z,R)\big) \big/ \Area\big(D^k(z,R)\big) \gtrapprox 2
\]
between the $k$-spheres and $k$-disk, under the hypothesis of \hyperref[lem:ACL-C]{Lemma~\ref{lem:ACL-C}}. Na\"ively, if we can replace the factors $2$ with some factors $\alpha_k \geq 2$ which grows with $k$, the estimate \hyperref[eqn:OriginalFactors]{(\ref*{eqn:OriginalFactors})} becomes
\begin{align*}
\Vol(B^n\ntts) \geq \bigg(\frac{R}{n} \cdot \alpha_{n-1}\bigg)\bigg(\frac{R}{n-1} \cdot \alpha_{n-2}\bigg) \cdots \bigg(\frac{R}{1} \cdot \alpha_0\bigg)= \frac{\prod_{i=0}^{n-1} \alpha_k}{2^n} \cdot \frac{(2R)^n}{n!}_,
\end{align*}
which is a better volume estimate than the one we obtained in \hyperref[thm:Closed-Balls]{Theorem~\ref*{thm:Closed-Balls}}. We should make the following observation before we extend our na\"ivet\'e any further.

\begin{obs} \label{obs:Naivete}
In the class of manifolds satisfying the hypothesis of \hyperref[thm:Closed-Balls]{Theorem~\ref*{thm:Closed-Balls}}, the constant factors $2$ in the estimate \hyperref[eqn:ACL-C]{(\ref{eqn:ACL-C})} of \hyperref[lem:ACL-C]{Lemma~\ref*{lem:ACL-C}} is the optimal universal constant factors in all dimensions.
\end{obs}

Let us elaborate on \hyperref[obs:Naivete]{Observation~\ref*{obs:Naivete}}. Consider the real projective space $\RP^{k+1}$ with the homogenous symmetric metric $g_\sph$ with constant curvature $\kappa\equiv1$. We have $\Sys(\RP^{k+1}\nts,g_\sph)=\pi$. Now, for any point $z \in \RP^{k+1}\nts$, take the symmetric hyperplane $\RP^k$ through $z$. Then, for the sphere $S^k(z,r) \subset \RP^{k+1}$ and the disk $D^k(z,r) \subset \RP^k$, with radius $r \in (0,\pi/2)$, we have
\[
\Area\big(S^k(z,r)\big) \big/ \Area\big(D^k(z,r)\big) \searrow 2 \quad \text{as} \quad r \nearrow \pi/2,
\]
since $S^k(z,r)$ and $D^k(z,r)$ become nearly isometric to the standard $k$-dimensional sphere and hemisphere respectively, as $r \rightarrow \pi/2$ from below. Hence, it is impossible to replace the factors 2 with an increasing sequence of constants $\alpha_k$, universal for the class of manifolds satisfying the hypothesis of \hyperref[thm:Closed-Balls]{Theorem~\ref*{thm:Closed-Balls}}.

\subsubsection*{Euclidean and Spherical Ratios}

\hyperref[obs:Naivete]{Observation~\ref{obs:Naivete}} suggests that, in order to make further improvements, one may seek a sequence $\alpha_k$ of sphere-disk area ratios bound, valid for any metric $g$ on a particular manifold $M$. For example, if $M \cong T^n\nts$, one may hope to find sphere-disk area ratio bounds $\alpha_k$ that grows comparably to the \emph{euclidean} sphere-disk area ratios:
\begin{align*}
\textstyle \alpha^\euc_k=2 \sqrt{\pi} \; \Gamma\!\left(\frac{k+2}{2}\right) \nts \big/ \, \Gamma\!\left(\frac{k+1}{2}\right)_.
\end{align*}
If we can find a sequence $\alpha_k$ with the same order of growth as the euclidean ones, then the argument of \hyperref[thm:Closed-Balls]{Theorem~\ref*{thm:Closed-Balls}} would produce isosystolic constants $C_n$ with the same order of the growth as the volume of euclidan balls, which resolves the so-called \emph{Generalized Geroch Conjecture} by Gromov affirmatively.

Alternatively, for each manifold $M$, one could also seek a sequence of sphere-disk area ratio bounds as \emph{functions} in the radius $r$, instead of constants. For example, if $M \cong \RP^n\nts$, one may hope for a sequence of sphere-disk area ratio functions that behaves comparably to the \emph{spherical} sphere-disk area ratio functions:
\begin{align*}
\alpha^{\sph}_k(r)=\frac{\sqrt{\pi} \; \Gamma \! \left(\frac{k}{2}\right)}{\Gamma \! \left(\frac{k+1}{2}\right) \int_{0}^{r} (\sin t)^{k-1} \, dt}
\end{align*}
where the metric is normalized so that $0 \leq r \leq \pi/2$. If we can replace the constant factor 2 in \hyperref[lem:ACL-C]{Lemma~\ref*{lem:ACL-C}} with functions $\alpha_k(r)$ comparable to $\alpha_k^\sph(r)$, then the argument of \hyperref[thm:Closed-Balls]{Theorem~\ref*{thm:Closed-Balls}} would produce isosystolic constants $C_n$ with the same order of the growth as the volume of spherical balls. If we can indeed replace the constant factor 2 in \hyperref[lem:ACL-C]{Lemma~\ref*{lem:ACL-C}} with functions $\alpha_k^\sph(r)$, then successive integrations show that the normalized volume of metric balls in $(\RP^n\nts,g)$ is at least the volume of metric balls in $(\RP^n\nts,g_\sph)$, which then implies that $(\RP^n\nts,g_\sph)$ actually realizes the optimal isosystolic constant. Searching for a sequence of sphere-disk ratio functions is related to \emph{Filling Volume Conjecture} by Gromov.

\section{Geometric Three-Manifolds} \label{sec:Geom3}

As we have noted, our \hyperref[thm:Closed]{Theorem~\ref*{thm:Closed}} specializes to Hebda's inequality for \emph{all} closed essential surfaces with $C_2=2$. On the other hand, the cohomological criterion in \hyperref[thm:Closed]{Theorem~\ref*{thm:Closed}}, or even \hyperref[thm:Open]{Theorem~\ref*{thm:Open}}, imposes a material constraint in higher dimensions; we should not expect our inequality to hold for all essential $n$-manifolds. As it turns out, our inequality holds for a quite large class of essential 3-manifolds. We start our discussion of 3-manifolds with \emph{geometric} 3-manifolds in this section.

One may note immediately that \hyperref[thm:CSF]{Theorem~\ref*{thm:CSF}} (for compact space forms) can be specialized to closed 3-manifolds which arise as euclidean or spherical space forms. The euclidean geometry and the spherical geometry are two of the \emph{eight model geometries} for closed 3-manifolds, in the sense of Thurston. Our isosystolic inequality holds more generally for \emph{most} closed 3-manifolds which admit one of the eight model geometries.

\begin{thm} \label{thm:Geom}
Let $M$ be a closed essential 3-manifold which admits one of eight model geometries. Suppose that $M$ is not a lens space $L(p,q)$ with odd order $p$. Then, for any riemannian metric $g$ on $M$,
\begin{align*}
\Sys(M,g)^3 &\leq 6 \, \Vol(M,g).
\end{align*}
\end{thm}

After reviewing the eight model geometries in \hyperref[ssec:8Geom]{\S\ref*{ssec:8Geom}}, we prove \hyperref[thm:Geom]{Theorem~\ref*{thm:Geom}} for two non-aspherical geometries in \hyperref[ssec:NonAspherical]{\S\ref*{ssec:NonAspherical}}, for four aspherical Seifert geometries in \hyperref[ssec:CircleBundles]{\S\ref*{ssec:CircleBundles}}, and for two remaining aspherical geometries in \hyperref[ssec:SurfaceBundles]{\S\ref*{ssec:SurfaceBundles}}.

\subsection{Eight Model Geometries} \label{ssec:8Geom}

In late 1970's, Thurston developed the notion of \emph{model geometries} and proposed \emph{Geometrization Conjecture}, which roughly says that, given a closed 3-manifold $M$, each piece in a natural topological decomposition of $M$ admits one of the eight model geometries. A \emph{model geometry} is a simply connected smooth manifold $X$ together with a smooth transitive action of a Lie group $G$ on $X$ with compact stabilizers; we further impose two additional conditions that (i) $G$ is maximal among groups acting smoothly and transitively on $X$, and (ii) there is at least one cocompact discrete subgroup of $G$. We write $\Geo$ for the manifold $X$ equipped with a left-invariant metric with respect to the action of $G$. A 3-manifold is said to \emph{admit a model geometry $\Geo$}, if it is diffeomorphic to a quotient of $\Geo$ by a discrete torsion-free subgroup of $G$. 

In dimension 2, it is classically known that there are three model geometries $\Sph^2\nts, \Euc^2\nts, \Hyp^2\nts$, up to normalization, and they are precisely the standard 2-dimensional riemannian universal space forms of constant curvature $\kappa=+1, 0, -1$, respectively; we call them the 2-dimensional spherical, euclidean, and hyperbolic geometries, respectively. In dimension 3, Thurston found that there are eight model geometries:
\[
\Sph^3\nts, \quad \Euc^3\nts, \quad \Hyp^3\nts, \ \quad \SR, \quad \HR, \quad \SLtwo, \quad \Nil, \quad \Sol.
\]
See \cite[\S3.8]{Thurston:Book} for the detail on the classification of model geometries.

The geometries $\Sph^3\nts, \Euc^3\nts, \Hyp^3$ are the standard riemannian universal space forms of constant curvature $\kappa=+1, 0, -1$ respectively; we call them the 3-dimensional spherical, euclidean, and hyperbolic geometries. The geometries $\SR$ and $\HR$ have the product geometry, arising from the 2-dimensional model geometries. $\SLtwo$ is the universal cover of $\SL_2(\bbsR)$, and it fibers over $\Hyp^2\nts$. $\Nil$ is the geometry of the Heisenberg group, and it fibers over $\Euc^2\nts$. Finally, $\Sol$ is the geometry of the solvable group with 0-dimensional point stabilizers. The geometries other than $\Sph^3$ and $\SR$ are all contractible, and hence every geometric manifold is aspherical unless it is modeled on $\Sph^3$ or $\SR$.

The geometries $\SR, \Euc^3, \HR$ have product structures, and the geometries $\Sph^3\nts, \Nil, \SLtwo$ fibers over $\Sph^2\nts, \Euc^2\nts, \Hyp^2\nts$, respectively. These six geometries are called \emph{Seifert geometries}, since a closed 3-manifold is a \emph{Seifert fibered space}, i.e. a circle bundle over a base 2-orbifold, if and only if it is modeled on one of these geometries. A description of Seifert fibered spaces from this viewpoint can be found in \cite{Scott:8geom}; other standard results on these manifolds can be found in \cite[Ch.VI]{Jaco:3mfld}, \cite[Ch.2]{Hatcher:3mfld}. Seifert fibered spaces with $\SR$ or $\Sph^3$ geometry have the base 2-orbifold with $\Sph^2$ geometry, and they are said to be of \emph{spherical type}. Seifert fibered spaces with $\Euc^3$ or $\Nil$ geometry have the base 2-orbifold with $\Euc^2$ geometry, and they are said to be of \emph{euclidean type}. Seifert fibered spaces with $\HR$ or $\SLtwo$ geometry have a base 2-orbifold with $\Hyp^2$ geometry, and they are said to be of \emph{hyperbolic type}.

\subsection{Non-Aspherical Geometries} \label{ssec:NonAspherical}

Let us first discuss two non-aspherical geometries, $\Sph^3$ and $\SR$. We shall see that, except for lens spaces with odd-order fundamental groups, our inequality holds for all essential manifolds.

\subsubsection*{$\Sph^3$ Geometry}

The classification of spherical 3-manifolds is classical; see \cite[\S3-4]{GLLUW} for example. By the resolution of Geometrization Conjecture, every closed orientable 3-manifold with finite fundamental group admits $\Sph^3$ geometry. The proof below elaborates on the 3-dimensional spherical case of \hyperref[thm:CSF]{Theorem~\ref*{thm:CSF}}.

\begin{proof}[Proof of {\hyperref[thm:Geom]{Theorem~\ref*{thm:Geom}}} for $\Sph^3$ geometry]
It is classically known that a point group $\varGamma$ acting on $\Sph^3$ is cyclic, or is a central extension of a dihedral, tetrahedral, octahedral, or icosahedral group by a cyclic group of even order. Hence, for a closed manifold $M$ with $\Sph^3$ geometry, the order of $\varGamma \cong \pi_1(M)$ is odd if and only if $\varGamma$ is trivial (yielding $S^3$) or odd-order cyclic (yielding $L(p,q)$ with odd order $p$). The result now follows from \hyperref[thm:CSF]{Theorem~\ref*{thm:CSF}}.
\end{proof}

One noteworthy example is the Poincar\'e dodecahedral space $\varSigma^3\nts$. It is an integral homology sphere with trivial first and second homology, with respect to any coefficient, but it satisfies the isosystolic inequality $\Sys(\varSigma,g)^3 \leq 6 \, \Vol(M,g)$.

\begin{rem}
It is known that all spherical 3-manifolds other than $S^3$ are essential, but our result does not handle lens spaces $L(p,q)$ with odd order $p$. Some approaches to isosystolic inequalities for such $L(p,q)$ are proposed in \cite{KKSSW}.
\end{rem}

\subsubsection*{$\SR$ Geometry}

Manifolds with $\SR$ geometry are well-understood, and there are only four of them: $S^2 \times S^1\nts$, $S^2 \rtimes S^1\nts$, $\RP^3 \csum \RP^3\nts$, $\RP^2 \times S^1\nts$. Here, $S^2 \rtimes S^1$ denotes the non-orientable $S^2$-bundle over $S^1\nts$, or equivalently the mapping torus of the antipodal map on $S^2\nts$. Among the four closed 3-manifolds witih $\SR$ geometry, $\RP^3 \csum \RP^3\nts$ and $\RP^2 \times S^1$ are essential while $S^2 \times S^1$ and $S^2 \rtimes S^1$ are inessential; we shall elaborate on this fact in \hyperref[prop:Ess1]{Proposition~\ref*{prop:Ess1}}.

\begin{rem}
It is well-known that no isosystolic inequality holds for $S^2 \times S^1$ or $S^2 \rtimes S^1\nts$. To see this, take a sequence $g_m$ of product metrics, obtained from the standard product metric by scaling down the metric in $S^2$ direction by a factor $1/m$ while keeping the metric in $S^1$ direction unchanged. Then, for this sequence of metrics, the volume goes down to 0 as $m$ approaches $\infty$, while the systole remains constant for all $m$; clearly, no isosystolic inequality holds.
\end{rem}

\begin{proof}[Proof of {\hyperref[thm:Geom]{Theorem~\ref*{thm:Geom}}} for $\SR$ geometry]
We shall explicitly find three surfaces with a single triple intersection point. For $\RP^3 \csum \RP^3\nts$, we can take three copies of $\RP^2$ in one connect-summand $\RP^3$; we can perturb these $\RP^2$ to have a single triple intersection point. For $\RP^2 \times S^1\nts$, take a transverse pair of homotopically non-trivial loops $\gamma_1 \simeq \gamma_2$ on a fiber $F_0:=\RP^2\nts$, with a single intersection point. Then, $F_0$, $S_1:=\gamma_1 \times S^1\nts$, $S_2:=\gamma_2 \times S^1$ give three surfaces with a single triple intersection point. The result now follows from \hyperref[thm:Closed]{Theorem~\ref*{thm:Closed}}.
\end{proof}

We note that the existence of an isosystolic inequality for $\RP^3 \csum \RP^3$ is not prevented by the absence of isosystolic inequalities for the double cover $S^2 \times S^1\nts$. The above sequence of metrics $g_m$ on $S^2 \times S^1$ descends to a metric $\bar g_m$ on $\RP^3 \csum \RP^3\nts$, and the shortest non-trivial geodesic loop $\gamma \subset (S^2 \times S^1\nts, g_m)$ maps injectively onto a non-trivial geodesic loop $\bar \gamma \subset (\RP^3 \times \RP^3\nts,\bar g_m)$ with $\Length(\bar \gamma)=\Length(\gamma)$. However, for large $m$, the systole $\Sys(\RP^3 \csum \RP^3\nts,\bar g_m)$ is not realized by $\bar \gamma$; instead, it is realized by the shortest geodesic loop on each copy of $\RP^2\nts$. In particular, the systole $\Sys(\RP^3 \csum \RP^3\nts,\bar g_m)$ decreases as $m \rightarrow \infty$.

\subsection{Circle Bundles over a Surface} \label{ssec:CircleBundles}

Let us now consider closed 3-manifolds that arise as circle bundles over a surface. Such a manifold is always a Seifert fibereded space whose base is a surface, i.e. a 2-orbifold without elliptic singular points. We shall see that these manifolds satisfy the criterion of \hyperref[thm:Closed]{Theorem~\ref*{thm:Closed}}.

\begin{prop} \label{prop:CircleBundles}
Let $M$ be a closed 3-manifold, with a finite-degree cover $\hat M$ homeoemorphic to a circle bundle over a surface. Suppose that the base surface of $\hat M$ is not $S^2\nts$. Then, for any riemannian metric $g$ on $M$,
\begin{align*}
\Sys(M,g)^3 &\leq 6 \, \Vol(M,g).
\end{align*}
\end{prop}

\begin{proof}
Since $\hat M$ is closed, its base $S_0$ is necessarily closed. Then, there exist two non-separating loops $\gamma_1, \gamma_2 \subset S_0$ which intersect once. Let $F_1, F_2$ be the preimages of $\gamma_1, \gamma_2$, respectively, under the bundle projection. Then, $S_0, F_1, F_2$ are three non-separating surfaces with a single triple intersection point. Hence, $M$ satisfies the criterion of \hyperref[thm:Closed]{Theorem~\ref{thm:Closed}}, and the isosystolic inequality for $M$ follows.
\end{proof}

\begin{rem}
When a circle bundle $\hat M$ has $\RP^2$ as the base surface, $\hat M$ is either $\RP^3 \csum \RP^3$ or $\RP^2 \times S^1$; we have already considered these manifolds in \hyperref[ssec:NonAspherical]{\S\ref*{ssec:NonAspherical}}.
\end{rem}

\subsubsection*{$\Euc^3$, $\Nil$, $\HR$, and $\SLtwo$ Geometries}

We now consider the remaining four Seifert fibered geometries, $\Euc^3$, $\Nil$, $\HR$, and $\SLtwo$. A closed 3-manifold $M$ is an aspherical Seifert fibered space if and only if $M$ admits one of these geometries. The classification of these manifolds is classical; see, for examples, \cite{Hatcher:3mfld}. We have already established our inequality for Seifert fibered spaces with $\Euc^3$ geometry in \hyperref[thm:CSF]{Theorem~\ref*{thm:CSF}}; we now show that our inequality holds for other aspherical Seifert fibered spaces as well.

\begin{proof}[Proof of {\hyperref[thm:Geom]{Theorem~\ref*{thm:Geom}}} for aspherical Seifert fibered geometries]
Every closed Seifert fibered space is a circle bundle over a closed 2-orbifold. A 2-orbifold is said to be \emph{bad} if it doesn't admit a 2-manifold cover, and is said to be \emph{good} otherwise. Bad 2-orbifolds are rare; all 2-orbifolds with $\Euc^2$ or $\Hyp^2$ geometries are known to be good.

If $M$ is a closed aspehrical Seifert fibered space, it is of euclidean type or hyperbolic type. The base 2-orbifold is a closed 2-orbifold with $\Euc^2$ or $\Hyp^2$ geometries, and it admits an orientable closed 2-manifold cover with the same geometry. We can pull back the Seifert fibering along this covering map, and obtain a closed 3-manifold $\hat M$ which is indeed a circle bundle over an orientable closed surface other than $S^2\nts$. Hence, the result now follows from \hyperref[prop:CircleBundles]{Proposition~\ref*{prop:CircleBundles}}.
\end{proof}

\subsection{Surface Bundles over a Circle} \label{ssec:SurfaceBundles}

Finally, we consider closed 3-manifolds that arise as surface bundles over a circle. We shall verify that these manifolds satisfy the criterion of \hyperref[thm:Closed]{Theorem~\ref{thm:Closed}}.

\begin{prop} \label{prop:SurfaceBundles}
Let $M$ be a closed 3-manifold, with a finite-degree cover $\hat M$ homeoemorphic to a surface bundle over the circle. Suppose that the fiber surface of $\hat M$ is not $S^2\nts$. Then, for any riemannian metric $g$ on $M$,
\begin{align*}
\Sys(M,g)^3 &\leq 6 \, \Vol(M,g).
\end{align*}
\end{prop}

\begin{proof}
Since $\hat M$ is closed, its fiber $F \not \cong S^2$ is necessarily closed. We regard $\hat M$ as a mapping torus of $f \in \Homeo(F)$, i.e. the quotient of $F \times I$ by the identification $(x,0) \sim (f(x),1)$. Note that $f$ induces an isomorphism $f_*$ on $H_1(F;\bbZtwo)$, permuting the nonzero elements of $H_1(F;\bbZtwo)$. Since $H_1(F;\bbZtwo)$ is finite, some power $(f_*)^k=(f^k\ntts)_*$ is the identity map on $H_1(F;\bbZtwo)$. Let $\hat M_k$ be the $k$-fold cyclic cover of $\hat M$ with respect to the fibering of $\hat M$, i.e. the mapping torus of $f^k \in \Homeo(F)$.

Since $F \not \cong S^2$ is closed, there exist two non-separating loops $\gamma_1$ and $\gamma_2$ on $F$ which intersect once. Now, we construct two surfaces $S_1$ and $S_2$ as follows. Note that $\gamma_i$ and $f^k(\gamma_i)$ are homologous since $[\gamma_i]=(f^k\ntts)_*[\gamma_i]=[f^k(\gamma_i)]$ in $H_1(F;\bbZtwo)$. So, they co-bound a 2-chain in $F$; hence, there is a 2-chain in $F \times I$ whose boundary is the union of $\gamma_i \times \{0\}$ and $f^k(\gamma_i) \times \{1\}$. This 2-chain can be regarded as a relative 2-cycle in $(F \times I, \partial(F \times I))$, which then defines a class in $H_2(F \times I, \partial (F \times I) ; \bbZtwo)$. Now, we have Poinca\'{e}-Lefschetz duality $H_2(F \times I, \partial (F \times I) ; \bbZtwo) \cong H^1(F \times I; \bbZtwo)$ and a bijection between $H^1(F \times I;\bbZtwo)$ and $[F \times I, \RP^\infty]$; hence, we see that the 2-chain connecting $\gamma_i \times \{0\}$ and $f^k(\gamma_i) \times \{1\}$ defines a homotopy class of a map from $F \times I$ into $\RP^\infty\nts$. Now, by the standard argument as in \hyperref[ssec:Open-Smooth]{\S\ref*{ssec:Open-Smooth}}, we see that this 2-chain is represented by a (possibly non-orientable) smooth embedded surface $S'_i$ in $F \times I$, with $\partial S'_i=(\gamma_i \times \{0\}) \cup (f^k(\gamma_i) \times \{1\})$. Finally, we let $S_i$ to be the image of $S'_i$ under the identification map $q: F \times I \mapsto \hat M_k$. We also let $F_0$ be the image of $F \times \{0\} \subset F \times I$ under the identification map $q$.

By construction, we have a covering map $\hat M_k \rightarrow M$ and three non-separating surfaces $F_0, S_1, S_2$ in $\hat M^k$ that intersect transversely at a single triple intersection point, which is the intersection of the image of $\gamma_1 \cap \gamma_2$ in $F_0$. Hence, $M$ satisfies the criterion of \hyperref[thm:Closed]{Theorem~\ref*{thm:Closed}}, and the isosystolic inequality for $M$ follows.
\end{proof}

\begin{rem}
For any coefficient $\bbA$, when $f \in \Homeo(F)$ induces the identity map on $H_1(F;\bbA)$, homology group and cohomology \emph{group} of the mapping torus of $f$ are isomorphic to those of the product $F \times S^1\nts$, but cohomology \emph{ring} need not be isomorphic to that of the product $F \times S^1\nts$. When $f$ represents an element of Torelli subgroup of the mapping class group $\Mod(F)$, where $F$ is an orientable surface with $\chi(F) <0$, the so-called \emph{Johnson homomorphism} measures the extent the $\bbZ$-cohomology ring fails to be isomorphic to that of $F \times I$; see \cite{Johnson} or \cite{Hain}.\end{rem}

\subsubsection*{$\Sol$ and $\Hyp^3$ geometries}

We now consider the remaining two geometries. For $\Sol$ geometry, our inequality can be deduced immediately from \hyperref[prop:SurfaceBundles]{Proposition~\ref*{prop:SurfaceBundles}}.

\begin{proof}[Proof of {\hyperref[thm:Geom]{Theorem~\ref*{thm:Geom}}} for $\Sol$ geometry]
A closed manifold with $\Sol$ geometry admits a finite-degree cover by a $T^2$-bundle over the circle. Hence, the result follows from \hyperref[prop:SurfaceBundles]{Proposition~\ref*{prop:SurfaceBundles}}.
\end{proof}

Closed manifolds with $\Hyp^3$ geometry, i.e. hyperbolic 3-manifolds, are much less understood in comparison to manifolds with other geometries. However, in the past several years, there has been tremendous amount of progress in understanding the structure of these manifolds. In particular, Agol recently resolved \emph{Virtual Haken Conjecture} and \emph{Virtual Fibering Conjecture} affirmatively \cite{Agol:VHC} for closed hyperbolic 3-manifolds.

\begin{proof}[Proof of {\hyperref[thm:Geom]{Theorem~\ref*{thm:Geom}}} for $\Hyp^3$ geometry]
Virtual Fibering Conjecture, in this context, states that every closed hyperbolic 3-manifold has a finite-degree cover that fibers over the circle. With the affirmative resolution \cite{Agol:VHC} of the conjecture, \hyperref[prop:SurfaceBundles]{Proposition~\ref*{prop:SurfaceBundles}} applies to all closed hyperbolic 3-manifolds.
\end{proof}

\begin{rem}
Generally, a surface bundle is \emph{not} necessarily geometric. In particular, \hyperref[prop:SurfaceBundles]{Proposition~\ref*{prop:SurfaceBundles}} applies to a more general class of manifolds beyond geometric manifolds. We discuss this aspect further in \hyperref[ssec:HyperbolicPiece]{\S\ref*{ssec:HyperbolicPiece}}.
\end{rem}

\section{Aspherical Three-Manifolds} \label{sec:Asph3}

In this section, we establish isosystolic inequalities for \emph{all} closed aspherical 3-manifolds. Our aim is to show that we can always apply \hyperref[thm:Open]{Theorem~\ref*{thm:Open}} to these manifolds. As we shall see, \hyperref[thm:Closed]{Theorem~\ref*{thm:Closed}} alone is not sufficient to establish our inequality; in some cases, we apply \hyperref[thm:Open]{Theorem~\ref*{thm:Open}} to a closed aspherical manifold and its infinite-degree non-compact cover.

\begin{thm} \label{thm:Asph}
Let $M$ be a closed aspherical 3-manifold. Then, for any riemannian metric $g$ on $M$,
\begin{align} \label{eqn:Asph}
\Sys(M,g)^3 &\leq 6 \, \Vol(M,g).
\end{align}
\end{thm}

For a geometric aspherical manifold $M$, the isosystolic inequality is already established in \hyperref[thm:Geom]{Theorem~\ref*{thm:Geom}} (for geometric 3-manifolds). Hence, we focus on non-geometric aspherical manifolds. These manifolds necessarily have a non-trivial \emph{JSJ decomposition}, with at least one \emph{JSJ piece} which is hyperbolic or Seifert fibered. We consider these two cases separately; manifolds with a hyperbolic piece is treated in \hyperref[prop:HyperbolicPiece]{Proposition~\ref*{prop:HyperbolicPiece}}, and manifolds with a Seifert piece is treated in \hyperref[prop:SeifertPiece]{Proposition~\ref*{prop:SeifertPiece}}. Immediately from these propositions, \hyperref[thm:Asph]{Theorem~\ref*{thm:Asph}} follows.

\subsection{JSJ decomposition} \label{ssec:JSJ}

In 1970's, a canonical decomposition of closed aspherical 3-manifolds was discovered by Jaco and Shalen \cite{Jaco-Shalen:JSJ}, and indepenently by Johannson \cite{Johannson:JSJ}. We briefly review this decomposition, now known the \emph{JSJ decomposition}. For brevity, let us restrict ourselves to orientable 3-manifolds. 

A version for an orientable closed aspherical 3-manifold $M$ that we consider here, also known as the \emph{torus decomposition}, gives a decomposition of $M$ into \emph{JSJ-pieces} by cutting $M$ along a unique canonical (possibly empty) collection of pairwise non-isotopic and pairwise disjoint incompressible tori, which we refer to as \emph{JSJ-tori}, so that each JSJ-piece is either a Seifert fibered space or an atoroidal manifold. Hence, by the resolution of Geometrization Conjecture, each JSJ-piece is either a Seifert fibered space or a manifold whose interior admits a hyperbolic structure; we simply call each JSJ-piece a Seifert piece or a hyperbolic piece respectively. The pairwise non-isotopic condition on JSJ-tori in this version implies that we don't use $T^2 \times I$ as a JSJ-piece unless $M$ is a $T^2$-bundle. The collection of JSJ-tori can be uniquely characterized as the \emph{minimal} collection of tori which decomposes $M$ into such pieces. In general, there are many vertical non-peripheral incompressible tori within Seifert pieces, but these do not appear as one of JSJ-tori.

Since $M$ is aspherical, each Seifert piece admits one of the four aspherical Seifert fibered geometries, i.e. $\Euc^3\nts$, $\Nil$, $\HR$, $\SLtwo$. As before, aspherical Seifert fibered spaces are said to be euclidean type or hyperbolic type depending on the geometry of the base 2-orbifolds; in other words, it is euclidean type if it has $\Euc^3$ or $\Nil$ geometries, and it is hyperbolic type if it has $\HR$ or $\SLtwo$ geometries.

If the collection of JSJ-tori is empty, then the closed manifold $M$ makes up a single JSJ-piece, and hence it is a closed aspherical Seifert fibered space (with geometries $\Euc^3\nts$, $\Nil$, $\HR$, or $\SLtwo$) or a closed hyperbolic manifolds (with geometry $\Hyp^3$). If the collection of JSJ-tori is non-empty, each JSJ-piece has at least one boundary torus. A Seifert piece with non-empty boundary tori is a circle-bundle over a 2-orbifold with non-empty boundary. For these Seifert pieces, just as for circle-bundles over 2-manifold with boundary, a section always exists, i.e. there is a horizontal surface transverse to all fibers. It follows that the geometry is $\Euc^3$ if it is of euclidean type, and $\HR$ if it is of hyperbolic type.

Among the Seifert pieces with non-empty boundary, the ones of euclidean type are quite special while the ones of hyperbolic type are generic. Let us describe these pieces of euclidean type. There are only three 2-orbifolds with $\Euc^2$ geometry with non-empty boundary, which are (i) an annulus, (ii) a M\"obius band, and (iii) a disk with two elliptic points of order 2; the corresponding orientable Seifert fibered spaces are (i) $T^2 \times I$, (ii) $K \ltimes I$, and (iii) $K \ltimes I$ respectively, where $K \ltimes I$ denotes the twisted $I$-bundle over a Klein bottle. Here, (ii) and (iii) give the homeomorphic underlying manifolds, but with different Seifert fibering. It is easy to see that (i) has $\Euc^3$ geometry; furthermore, since (ii) and (iii) are both covered by (i), they also have $\Euc^3$ geometry. Hence, there are only two Seifert pieces of euclidean type with boundary, up to homeomorphism.

We have only mentioned five aspherical geomtries ($\Euc^3\nts$, $\Nil$, $\HR$, $\SLtwo$, and $\Hyp^3$) explicitly in the above discussion so far. Aspherical manifolds with the last aspherical geometry, i.e. $\Sol$ geometry, are rather peculiar in the context of the JSJ decomposition. If $M$ admits $\Sol$ geometry, it is either a $T^2$-bundle over the circle or a $T^2$-semibundle over the interval; there is a single JSJ-torus, given by a fiber torus, which decomposes $M$ into a single copy of $T^2 \times I$ or two copies of $K \ltimes I$, repsectively. So, JSJ-pieces of $M$ consist only of Seifert pieces of euclidean type. The converse is also true. Suppose that JSJ-pieces of $M$ consist only of Seifert pieces of euclidean type. If $T^2 \times I$ appears as a JSJ-piece, then the two boundary tori must be identified since JSJ-tori must be pairwise non-isotopic by definition, and hence $M$ must be a $T^2$-bundle over the circle. If copies of $K \ltimes I$ appear as the only JSJ-pieces, $M$ must be a union of precisely two copies of $K \ltimes I$, and hence $M$ is a $T^2$-semibundle over the interval. In either case, $M$ must admit $\Nil$, $\Euc^3\nts$, or $\Sol$ geometry. However, if $M$ admits one of the Seifert geometries, the fiber is not a JSJ-torus by definition; hence $M$ must admit $\Sol$ geometry.

The isosystolic inequality \hyperref[eqn:Asph]{(\ref*{eqn:Asph})} for geometric manifolds is already established in \hyperref[thm:Geom]{Theorem~\ref*{thm:Geom}}. Our goal for the rest of this section is to establish the inequality for non-geometric aspherical manifolds. From the discussion above, the JSJ decomposition of non-geometric aspherical manifolds must be non-trivial, and the collection of JSJ-pieces must contain a hyperbolic piece or a Seifert piece of hyperbolic type.

A closed aspherical manifold $M$ is said to be a \emph{graph manifold} if all JSJ-pieces are Seifert pieces. A graph-manifold is said to be a \emph{trivial} graph-manifold if the collection of JSJ-tori is empty, and is said to be a \emph{non-trivial} graph-manifold otherwise. Trivial graph-manifolds are just aspherical Seifert fibered spaces. Non-trivial graph-manifolds admit $\Sol$ geometry if JSJ-pieces are all euclidean type.

\subsection{Hyperbolic Piece} \label{ssec:HyperbolicPiece}

We first consider non-geometric aspherical manifolds with at least one hyperbolic piece in the JSJ decomposition. The proof below relies on the work of Przytycki and Wise \cite{Przytycki-Wise} on non-positively curved manifolds; an alternative proof is sketched at the end of \hyperref[ssec:SeifertPiece]{\S\ref*{ssec:SeifertPiece}}.

\begin{prop} \label{prop:HyperbolicPiece}
Let $M$ be a closed aspherical 3-manifold. Suppose that $M$ is not geometric and that it has at least one hyperbolic piece in its JSJ decomposition. Then, for any riemannian metric $g$ on $M$,
\begin{align*}
\Sys(M,g)^3 &\leq 6 \, \Vol(M,g).
\end{align*}
\end{prop}

\hyperref[prop:HyperbolicPiece]{Proposition~\ref*{prop:HyperbolicPiece}} is a special case of \hyperref[prop:NPC]{Proposition~\ref*{prop:NPC}} below, which establishes the isosystolic inequality more general class of closed \emph{non-positively curved} manifolds. We adapt a convention that a closed manifold is said to be \emph{non-positively curved} if it \emph{admits} a riemannian metric of non-positive sectional curvature. Such a manifold is always aspherical by Cartan-Hadamard theorem; see, for example, \cite{doCarmo}. Let us first provide some backgrounds on closed non-positively curved 3-manifolds.

A closed geometric aspherical 3-manifold is non-positively curved if and only if it admits $\Euc^3\nts$, $\HR$, or $\Hyp^3$ geometry \cite{Gromoll-Wolf}. Among a closed non-geometric aspherical 3-manifolds, non-positively curved manifolds are common; the work of Leeb \cite{Leeb} shows that any closed aspherical 3-manifold with at least one hyperbolic piece in the JSJ decomposition is non-positively curved. There are non-positively curved non-trivial graph manifolds as well; however, it should be emphasized that there are also non-trivial graph manifolds that are not non-positively curved \cite{Leeb}.

The relevance of non-positively curved manifolds in our discussion of isosystolic inequality is that they are now known to fiber virtually over the circle. For closed non-positively curved graph-manifolds, this fact was shown by Svetlov \cite{Svetlov}; see \cite{Wang-Yu} and \cite{Liu} for related results. For closed non-geometric non-positively curved manifolds with at least one hyperbolic piece in the JSJ decomposition, this fact can be deduced as follows. First, Przytycki and Wise \cite{Przytycki-Wise} recently announced that, for such a manifold $M$, $\pi_1(M)$ is \emph{virtually special}, i.e. $\pi_1(M)$ virtually injects into a finitely generated right-angled Artin group. Second, a finitely generated right-angled Artin group injects into a finitely-generated right-angled Coxeter groups; see \cite{Hsu-Wise:Graph}, \cite{Davis-Januszkiewicz}. Third, Agol \cite{Agol:RFRS} showed that finitely-generated right-angled Coxeter groups are virtually \emph{residually finite rationally solvable} (RFRS); so, if $\pi_1(M)$ is virtually special, it is virtually RFRS. Fourth, and most importantly, Agol \cite{Agol:RFRS} showed that a closed aspherical 3-manifold $M$ fibers virtually over the circle if $\pi_1(M)$ is virtually RFRS.

\begin{prop} \label{prop:NPC}
Let $M$ be a closed aspherical 3-manifold. Suppose that $M$ is not geometric and that it is non-positively curved, e.g. $M$ has at least one hyperbolic piece in its JSJ decomposition. Then, for any riemannian metric $g$ on $M$,
\begin{align*}
\Sys(M,g)^3 &\leq 6 \, \Vol(M,g).
\end{align*}
\end{prop}

\begin{proof}
Since non-positively curved manifolds fibers virtually over the circle, the isosystolic inequality follows immediately from \hyperref[prop:SurfaceBundles]{Propsition~\ref{prop:SurfaceBundles}}.
\end{proof}

\subsection{Tame Manifolds} \label{ssec:Tame}

Before we proceed, let us make an observations on homology and cohomology of a \emph{tame} manifold. Recall that a non-compact manifold $M$ is said to be \emph{tame} if it is homeomorphic to the interior of a compact manifold $\bar M$ with boundary $\partial \bar M$. The inclusion $i:M \hookrightarrow \bar M$ is a homotopy equivalence, inducing isomorphsims
$i_* : H_*^\Ksupp(M;\bbA) \stackrel{\cong}{\longrightarrow} H_*(\bar M;\bbA)$ and
$i^* : H^*(\bar M;\bbA) \stackrel{\cong}{\longrightarrow} H^*_\Fsupp(M;\bbA)$.
On the other hand, $\bar M$ embeds in $M$ homeomorphically onto $\bar M \ssminus N \subset M$ where $N$ is a regular neighborhood of $\partial \bar M$ in $\bar M$; so, we have an embedding of a pair $j:(\bar M, \partial \bar M) \rightarrow (M, M \cap \bar N)$, where $\bar N$ is the closure of $N$. Taking the suitable direct limits, we obtain canonical isomorphisms
$j_* : H_*(\bar M,\partial \bar M;\bbA) \stackrel{\cong}{\longrightarrow} H_*^\Fsupp(M;\bbA)$ and
$j^* : H^*_\Ksupp(M;\bbA) \stackrel{\cong}{\longrightarrow} H^*(\bar M, \partial \bar M;\bbA)$.
Now, we recall two versions of the Poincar\'e-Lefschetz duality: for any compact $n$-manifold $M$ with boundary, we have
\begin{align*}
\PLD & : H^k(\bar M,\partial \bar M;\bbZtwo) \stackrel{\cong}{\longrightarrow} H_{n-k}(\bar M;\bbZtwo), \\
\PLD & : H^k(\bar M;\bbZtwo) \stackrel{\cong}{\longrightarrow} H_{n-k}(\bar M, \partial \bar M;\bbZtwo),
\end{align*}
where the coefficient $\bbZtwo$ can be replaced with $\bbZ$ if $M$ is orientable. In this setting, Poincar\'e duality of a tame manifold $M$ naturally corresponds to Poincar\'e-Lefschetz duality of $\bar M$ in the sense that the following diagrams commute:
\begin{align*}
\begin{CD}
H^k_\Ksupp(M;\bbZtwo) @> \PD > \cong > H_{n-k}^\Ksupp(M;\bbZtwo)\\
@V j^* V \cong V @V \cong V i_* V \\
H^k(\bar M,\partial \bar M;\bbZtwo) @> \cong > \PLD > H_{n-k}(\bar M;\bbZtwo)
\end{CD}
\end{align*}
\begin{align*}
\begin{CD}
H^k_\Fsupp(M;\bbZtwo) @> \PD > \cong > H_{n-k}^\Fsupp(M;\bbZtwo)\\
@A i_* A \cong A @A \cong A j^* A \\
H^k(\bar M;\bbZtwo) @> \cong > \PLD > H_{n-k}(\bar M,\partial \bar M;\bbZtwo)
\end{CD}
\end{align*}

\begin{rem}
In both diagrams above, as usual, the coefficient $\bbZtwo$ can be replaced with $\bbZ$ if $M$ is orientable. We only utilize the correspondence with $\bbZtwo$ coefficient.
\end{rem}

\subsection{Seifert Piece} \label{ssec:SeifertPiece}

We now consider closed non-geometric aspherical manifolds with at least one Seifert piece of hyperbolic type in the JSJ decomposition. This class of manifolds include all non-geometric \emph{graph-manifolds}. 

\begin{prop} \label{prop:SeifertPiece}
Let $M$ be a closed aspherical 3-manifold. Suppose that $M$ is non-geometric and that it has at least one Seifert piece of hyperbolic type in its JSJ decomposition. Then, for any riemannian metric $g$ on $M$,
\begin{align} \label{eqn:SeifertPiece}
\Sys(M,g)^3 &\leq 6 \, \Vol(M,g).
\end{align}
\end{prop}

\begin{proof}
We first prove the proposition for orientable manifold $M$. Let $N$ be a compact Seifert piece of hyperbolic type. By the non-geometric assumption, it has a non-empty boundary tori. Since $N$ is a Seifert piece of hyperbolic type, $N$ admits a finite-degree covering by a circle bundle over a base surface with $\Hyp^2$ geometry. If this base surface has genus 0, we can take a further finite-degree covering by a circle bundle over a base surface with positive genus and with $\Hyp^2$ geometry. Let us write $\hat N$ for such a cover, and $S_0$ for its base surface; $\hat N$ is a compact manifold such that $\partial \hat N \neq \nil$ is a union of tori, and $S_0$ is a compact surface with $\partial S_0 \neq \nil$. Then, we may follow the argument of \hyperref[prop:CircleBundles]{Proposition~\ref*{prop:CircleBundles}} and find vertical tori $F_1$, $F_2$, in $\hat N$ so that $S_0$, $F_1$, $F_2$ are three surfaces in $\hat N$ with a single triple intersection point; hence, they represent non-zero classes $\bsS_0 \in H_2(\hat N, \partial \hat N;\bbZtwo)$ and $\bsF_1, \bsF_2 \in H_2(\hat N;\bbZtwo)$ with non-zero triple algebraic intersection number.

We have injections $\pi_1(\hat N) \hookrightarrow \pi_1(N) \hookrightarrow \pi_1(M)$; the first injection has finite index, but the second injection always has infinite index. We take an infinite-degree cover $\hat M$ of $M$, corresponding to the image of $\pi_1(\hat N)$ in $\pi_1(M)$. Note that $\hat M$ is a non-compact manifold which contains an isometric copy of $\hat N$, such that the inclusion $\hat N \hookrightarrow \hat M$ is a homotopy equivalence. Furthermore, by the work of Simon \cite{Simon:Tame}, we know that $\hat M$ is a tame manifold, homeomorphic to $\interior(\hat N)$ and deformation retracting onto $\hat N$. In particular, we have $H_*^\Fsupp(\hat M;\bbZtwo) \cong H_*^\Fsupp(\interior (\hat N);\bbZtwo) \cong H_*(\hat N,\partial \hat N;\bbZtwo)$ and $H_*^\Ksupp(\hat M;\bbZtwo) \cong H_*^\Ksupp(\interior (\hat N);\bbZtwo) \cong H_*(\hat N;\bbZtwo)$

Now, let $\bsS'_0 \in H_2^\Fsupp(\hat M;\bbZtwo)$ and $\bsF'_1, \bsF'_2 \in H_2^\Ksupp(\hat M;\bbZtwo)$ be the classes corresponding to $\bsS_0 \in H_2(\hat N,\partial \hat N;\bbZtwo)$, $\bsF_1, \bsF_2 \in H_2(\hat N;\bbZtwo)$ via the above isomorphisms. Note that $\bsF'_1$ and $\bsF'_2$ are represented by $F_1$ and $F_2$. For $\bsS'_0$, we can find a representative $S_0'$ by stretching out the ends of the surface $S_0$ by reversing the deformation retract of $\hat M$ onto $\hat N$. Hence, we have surfaces $S'_0, F_1, F_2$ with a single triple intersection point, representing $\bsS'_0, \bsF'_1, \bsF'_2$ respectively. It follows that the Poincar\'e dual classes $\bssigma_0 \in H^1_\Fsupp(\hat M;\bbZtwo)$, $\bsvarphi_1, \bsvarphi_2 \in H^1_\Ksupp(M;\bbZtwo)$ of $\bsS'_0, \bsF'_1, \bsF'_2$ satisfy $\bssigma_0 \cup \bsvarphi_1 \cup \bsvarphi_2 \neq \bszero$ in $H^3_\Ksupp(\hat M;\bbZtwo)$. Hence, $\hat M$ satisfies the hypothesis of \hyperref[thm:Open]{Theorem~\ref*{thm:Open}}, and the isosystolic inequality \hyperref[eqn:SeifertPiece]{(\ref*{eqn:SeifertPiece})} follows.

Finally, if $M$ is non-orientable, we may first take an orientable double cover, and then follow the above argument to find $\hat M$ that satisfy the hypothesis of \hyperref[thm:Open]{Theorem~\ref*{thm:Open}}; the isosystolic inequality \hyperref[eqn:SeifertPiece]{(\ref*{eqn:SeifertPiece})} follows in this case as well.
\end{proof}

\begin{rem}
One can also prove \hyperref[prop:HyperbolicPiece]{Proposition~\ref*{prop:HyperbolicPiece}}, i.e. an isosystolic inequality for non-geometric aspherical manifold with a hyperbolic piece, by combining the arguments in the proof of \hyperref[prop:SurfaceBundles]{Proposition~\ref*{prop:SurfaceBundles}} and \hyperref[prop:SeifertPiece]{Proposition~\ref*{prop:SeifertPiece}}. We give a brief sketch here.

If a non-geometric aspherical manifold $M$ has a hyperbolic piece $N$ in its JSJ decomposition, we can lift it to a cover corresponding to this hyperbolic piece. This non-compact cover is tame by the work of Simon \cite{Simon:Tame}, so it is homeomorphic to the interior of $N$ and admits a hyperbolic metric. Then, by the work of Wise on Virtual Fibering Conjecture \cite{Wise:VFC} for cusped hyperbolic manifolds, we can further lift it to a cover which fibers over the circle. Taking a cyclic cover if necessary, we see that the original manifold $M$ is covered by a surface bundle $\hat M$ whose gluing map acts trivially on homology $H_1^\Ksupp(F;\bbZtwo)$ and $H_1^\Fsupp(F;\bbZtwo)$ of the fiber.

It remains to construct three surfaces in $\hat M$ with triple intersection point, representing non-zero homology classes and with at least one of them having compact support. We take one surface to be the non-compact fiber surface $F_0$. Two more surfaces $S_1$, $S_2$ can be constructed from cobordisms in $F \times I$. If the genus of $F$ is positive, we can construct closed surfaces $S_1,S_2$ from a pair of loops on $F$ with a single intersection point between them. If the genus of $F$ is zero, the fiber is a sphere with at least four punctures, and we can construct closed surface $S_1$ and a non-compact surface $S_2$ from a loop around a puncture and an arc between this puncture to another puncture. 
\end{rem}

\section{Essential Three-Manifolds} \label{sec:Ess3}

In this section, we establish our isosystolic inequality for \emph{most} closed 3-manifolds. Let us write $V_0:=S^3\nts$, $V_k:=\csum_k(S^2 \times S^1\ntts)$ and $V_{-k}:=\csum_k(S^2 \rtimes S^1\ntts)$ for $k>0$, where $\csum_k$ denotes the $k$-times repeated connected sum.

\begin{thm} \label{thm:Ess}
Let $M$ be a closed 3-manifold. Suppose that $M$ is not $V_k$ or a connected sum $L(p_1,q_1)\csum \cdots \csum L(p_m,q_m) \csum V_k$ with odd orders $p_1, \cdots, p_m$ for some integer $k$. Then, for any riemannian metric $g$ on $M$,
\begin{align*}
\Sys(M,g)^3 &\leq 6 \, \Vol(M,g).
\end{align*}
\end{thm}

It should be noted that, among the closed 3-manifolds excluded in \hyperref[thm:Ess]{Theorem~\ref*{thm:Ess}}, the ones of the form $L(p_1,q_1)\csum \cdots \csum L(p_m,q_m) \csum V_k$ are essential, and hence some isosystolic inequality must holds by Gromov's \hyperref[thm:Gromov]{Theorem~\ref*{thm:Gromov}}; since $\bbZtwo$-homology is useless to study odd-order lens spaces, one must employ a different tools to establish a better isosystolic inequality for these manifolds. It is not hard to see that the remaining manifolds $V_k$, on the other hand, accommodates no isosystolic inequality.

After reviewing the prime decomposition for 3-manifolds in \hyperref[ssec:Prime]{\S\ref*{ssec:Prime}} and identifying the class of essential 3-manifolds explicitly in \hyperref[ssec:ConnSumEss]{\S\ref*{ssec:ConnSumEss}}, we prove \hyperref[thm:Ess]{Theorem~\ref*{thm:Ess}} in \hyperref[ssec:ConnSumIsosys]{\S\ref*{ssec:ConnSumIsosys}}. Our proof simply combines the results from previous sections with an observation on the relationship between the connected sum operation and the homological criterion in \hyperref[thm:Open]{Theorem~\ref*{thm:Open}}. We also record in \hyperref[ssec:ConnSumIsosys]{\S\ref*{ssec:ConnSumIsosys}} that a closed 3-manifold $M$ is essential only if some isosystolic inequality holds for $M$. This is the converse of Gromov's \hyperref[thm:Gromov]{Theorem~\ref*{thm:Gromov}} in dimension 3; for \emph{orientable} manifolds, the converse of \hyperref[thm:Gromov]{Theorem~\ref*{thm:Gromov}} has already been established by the work of Babenko \cite{Babenko:Asymptotic} in all dimensions. 

\subsection{Prime Decomposition} \label{ssec:Prime}

Recall that a closed 3-manifold $M$ is said to be \emph{prime} if a connected sum decomposition $M=P \csum Q$ implies $P\cong S^3$ or $Q \cong S^3\nts$, and it is said to be \emph{irreducible} if every embedded 2-sphere bounds a 3-ball. It is classically known that a closed prime 3-manifold is either (i) an irreducible manifold, (ii-a) $S^2 \times S^1\nts$, or (ii-b) $S^2 \rtimes S^1$; moreover, by the resolution of Geometrization Conjecture, we also know that a closed irreducible manifold is (i-a) an aspherical manifold, (i-b) $\RP^2 \times S^1\nts$, (i-c) a spherical space form other than $S^3\nts$, or (i-d) $S^3\nts$.

The \emph{prime decomposition} is a canonical decomposition of a compact 3-manifold \cite{Kneser}, \cite{Milnor:Prime}. For orientable closed 3-manifolds, the prime decomposition theorem states that every orientable closed 3-manifold $M$ can be expressed as a connected sum of (i) irreducible manifolds and (ii-a) copies of $S^2 \times S^1\nts$, and this expression is unique up to reordering. For non-orientable closed 3-manifolds, the prime decomposition theorem states that every non-orientable closed 3-manifold $M$ can be expressed as a connected sum of (i) irreducible manifolds and (ii-b) copies of $S^2 \rtimes S^1\nts$, and this expression is unique up to reordering as long as we insist that each summand is either an irreducible manifold or $S^2 \rtimes S^1$; we remark that the restrictions on summands is necessary in non-orientable case, since $\RP^3\csum(S^2 \times S^1\ntts)=S^2 \rtimes S^1\nts$.

\subsection{Connected Sum and Essential Manifolds} \label{ssec:ConnSumEss}

Suppose that $M$ is a closed $n$-manifold. Let us first recall that, if $q:M \rightarrow Q$ is a degree-1 map and $Q$ is essential, then $M$ is also essential \cite[\S0]{Gromov:Filling}. More specifically, if $K$ is an aspherical space and a map $f: Q \rightarrow K$ represents a non-trivial class of $K$, i.e. $f_*[Q]$ defines a non-zero homology class in $K$, then $f_*q_*[M]=\pm f_*[Q]$ is non-zero. In particular, taking a degree-1 map $q: Q \csum R \rightarrow Q$ that collapses the summand $R$, we have the following observation \cite[\S0]{Gromov:Filling}.

\begin{obsalpha}[Gromov] \label{obs:ConnSum}
If $M$ is a closed $n$-manifold such that $M=Q \csum R$ and $Q$ is essential, then $M$ is also essential.
\end{obsalpha}

\begin{rem}
More generally, if $q:M \rightarrow Q$ is a degree-$d$ map with $d \neq 0$ and $Q$ is essential, then it is \emph{often} true that $M$ is also essential; for $M$ to be essential, it is necessary and sufficient that the order of $f_*[Q]$ in homology of $K(\pi_1(Q),1)$ does not divide the degree $d$ of the map $q$. It should be noted that this restriction is necessary since there are examples of essential manifolds $Q$ with an inessential finite-degree cover $M$; examples exist in all dimensions ($Q=\RP^n$ and $M=S^n$), and there are infinite number of examples already in 3-dimensions ($Q=L(p,q)$ and $M=S^3$). It is worthwhile to note that there are also examples of an essential manifold $Q$ with an inessential finite-degree cover $M$ with non-trivial $\pi_1(M)$; for example, if $Q'$ is an essential manifold with a compact universal cover $M'$, then $Q=Q' \times S^1$ is an essential manifold with an inessential finite-degree cover $M=M' \times S^1$ with $\pi_1(M) \cong \bbZ$. There are also less common examples of an essential manifold $Q$ that doesn't fiber over $S^1$ with an inessential finite-degree cover $M$; for example, $Q=\RP^3 \csum \RP^3$ is an essential manifold with an inessential double cover $M=S^2 \times S^1\nts$.
\end{rem}

In dimension 2, a closed surface is essential if and only if it is not $S^2\nts$. Using \hyperref[obs:ConnSum]{Observation~\ref*{obs:ConnSum}}, we can articulate which closed 3-manifolds are essential. Recall that we write $V_0=S^3\nts$, $V_k=\csum_k(S^2 \times S^1\ntts)$ and $V_{-k}=\csum_k(S^2 \rtimes S^1\ntts)$ for $k>0$.

\begin{prop} \label{prop:Ess1}
A closed 3-manifold $M$ is essential if and only if $M \not \in \{ V_k \}_{k \in \bbZ}$.
\end{prop}

\begin{rem}
A manifold, to which \hyperref[thm:Ess]{Theorem~\ref*{thm:Ess}} applies, is essential by \hyperref[prop:Ess1]{Proposition~\ref*{prop:Ess1}}.
\end{rem}

\begin{proof}
Clearly, $V_0=S^3$ is inessential. For non-zero $k \in \bbZ$, $\pi_1(V_k)$ is a rank $k$ free group. Up to homotopy, any map from $V_k$ into an aspherical space $K$ factors through a map into $K(\pi_1(V_k),1) \simeq (S^1\ntts)^{\vee k}$, which has trivial homology $H_i(V_k;\bbA)$ for $i \geq 2$; hence, $V_k$ is inessential. Thus, if $M$ is essential, $M \neq V_k$ for any $k \in \bbZ$.

Now, suppose $M \neq V_k$ for any $k \in \bbZ$. Then, the prime decomposition of $M$ must contain an irreducible summand $P \not \cong S^3\nts$, i.e. $P$ is either (i-a) an aspherical manifold, (i-b) $\RP^2 \times S^1\nts$, or (i-c) a spherical space form other than $S^3\nts$. Aspherical manifolds are clearly essential. $\RP^2 \times S^1$ is essential via the inclusion $\RP^2 \times S^1 \hookrightarrow \RP^\infty \times S^1\nts$. It is also known from cohomology of finite groups that spherical manifolds other than $S^3$ are essential. Hence, the summand $P$ is essential, and hence $M$ is also essential by \hyperref[obs:ConnSum]{Observation~\ref*{obs:ConnSum}}.
\end{proof}

\subsection{Connected Sum and Isosystolic Inequality} \label{ssec:ConnSumIsosys}

We now discuss the relationship between the connected sum operation and our isosystolic inequality. We shall first prove \hyperref[thm:Ess]{Theorem~\ref*{thm:Ess}}; the proof relies on the following observation, which refines \hyperref[obs:ConnSum]{Observation~\ref*{obs:ConnSum}} for the class of manifolds that are relevant to us.

\begin{lem} \label{lem:MV}
If $M$ is a closed $n$-manifold such that $M=Q \csum R$ and $Q$ satisfies the hypothesis of \hyperref[thm:Open]{Theorem~\ref*{thm:Open}}, then $M$ also satisfies the hypothesis of \hyperref[thm:Open]{Theorem~\ref*{thm:Open}}.
\end{lem}

\begin{proof}
Suppose that $M=Q \csum R$ is obtained from gluing $Q^*:=Q \ssminus B(x)$ and $R^*:=R \ssminus B(y)$ along their bounary, where $B(x) \subset Q$ and $B(y) \subset R$ are open regular neighborhoods of $x \in Q$ and $y \in R$ respectively. Let $\hat Q^* \subset \hat Q$ be the lift of $Q^*$; note that $\hat Q^*=\hat Q-\bigcup_\alpha B(\hat x_\alpha)$ and $\partial \hat Q^*=\bigcup_\alpha \partial \bar B(\hat x_\alpha)$, where $B(\hat x_\alpha)$ is the lift of $B(x)$ around the lift $\hat x_\alpha$ of $x$. Gluing a copy $R^*_\alpha \cong R^*$ to each boundary component $\partial \bar B(\hat x_\alpha)$ of $\hat Q^*$, we obtain a cover $\hat M=\hat Q^* \cup \bigcup_\alpha R^*_\alpha$, which is a connected sum of $\hat Q$ and copies $R_\alpha$ of $R$. The lemma is just a consequence of Mayer-Vietoris sequence; we remark that this argument is valid even if the degree of the cover is infinite.
\end{proof}

\begin{proof}[Proof of {\hyperref[thm:Ess]{Theorem~\ref*{thm:Ess}}}]
Let $M$ be a closed 3-manifold given in the hypothesis. The prime decomposition of $M$ contains at least one summand $P$ which is not homeomorphic to $S^2 \times S^1$, $S^2 \rtimes S^1$, or a lens space $L(p,q)$ with odd-order $p$. It follows from the discussions in \hyperref[ssec:Prime]{\S\ref*{ssec:Prime}} and \hyperref[ssec:NonAspherical]{\S\ref*{ssec:NonAspherical}} that $P$ is a closed irreducible 3-manifold which is either (i-a) an aspherical manifold, (i-b) $\RP^2 \times S^1\nts$, or (i-$\mathrm{c}'$) a spherical space form with even-order fundamental group. Note that, from the proof of \hyperref[thm:Geom]{Theorem~\ref*{thm:Geom}} and \hyperref[thm:Asph]{Theorem~\ref*{thm:Asph}}, such a summand $P$ admits a cover $\hat P$ that satisfies the hypothesis of \hyperref[thm:Open]{Theorem~\ref*{thm:Open}}. Hence, by \hyperref[lem:MV]{Lemma~\ref*{lem:MV}}, $M$ admits a cover $\hat M$ that satisfies the hypothesis of \hyperref[thm:Open]{Theorem~\ref*{thm:Open}}. The desired inequality follows.
\end{proof}

Recall that, by the work of Gromov (\hyperref[thm:Gromov]{Theorem~\ref*{thm:Gromov}}), every essential $n$-manifold satisfies an isosystolic inequality; conversely, by the work of Babenko \cite{Babenko:Asymptotic}, an \emph{orientable} $n$-manifold is essential if it satisfies some isosystolic inequality. It is conceivable that the analogue of Babenko's theorem holds for non-orientable $n$-manifolds, so that a non-orientable $n$-manifold $M$ is essential if it satisfies some isosystolic inequality. In dimension 3, this statement can be verified directly.

\begin{prop} \label{prop:Ess2}
A closed 3-manifold is essential if and only if it satisfies some isosystolic inequality.
\end{prop}

\begin{proof}
By the work of Gromov and Babenko mentioned above, it suffices to show that no isosystolic inequality holds for closed non-orientable inessential 3-manifolds. By \hyperref[prop:Ess1]{Proposition~\ref*{prop:Ess1}}, these manifolds are of the form $V_{-k}=\csum_k(S^2 \rtimes S^1\ntts)$ for some positive integer $k$. As noted in \hyperref[ssec:NonAspherical]{\S\ref*{ssec:NonAspherical}}, there is a sequence $g_m$ of metrics on $V_{-1}=S^2 \rtimes S^1$ such that $\Vol(V_{-1},g_m) \rightarrow 0$ as $m \rightarrow \infty$ while $\Sys(V_{-1},g_m)$ remains constant for all $m$. Taking the connected sum of $(V_{-1},g_m)$ and smoothing out the metric, we obtain a sequence $g'_m$ of metrics on $V_{-k}$ such that $\Vol(V_{-k},g'_m) \rightarrow 0$ as $m \rightarrow \infty$ while $\Sys(V_{-k},g'_m)$ remains constant for all $m$. Hence, no isosystolic inequality holds for $V_{-k}$.
\end{proof}

\hyperref[thm:Ess]{Theorem~\ref*{thm:Ess}} should be viewed in the light of \hyperref[prop:Ess1]{Proposition~\ref*{prop:Ess1}} and \hyperref[prop:Ess2]{Proposition~\ref*{prop:Ess2}}. These propositions together states that a closed manifold $M$ is essential if and only if $M \not \in \{V_k\}_{k \in \bbZ}$, if and only if some isosystolic inequality holds for $M$. For most of these essential manifolds, \hyperref[thm:Ess]{Theorem~\ref*{thm:Ess}} establishes an inequality with an isosystolic constant $C=6$. Remaining essential manifolds excluded in \hyperref[thm:Ess]{Theorem~\ref*{thm:Ess}} are of the form $L(p_1,q_1) \csum \cdots \csum L(p_m,q_m) \csum V_k$ with $m \geq 1$, odd orders $p_i$, and an integer $k$. 

It is desirable to obtain an inequality with a reasonably small isosystolic constant for these remaining manifolds. If the inessential summand $V_k$ is trivial, i.e. $k=0$, then the work of Sabourau \cite[Thm.3.1]{Sabourau} shows that the optimal isosystolic constant for a connected sum $L(p_1,q_1) \csum \cdots \csum L(p_m,q_m)$ has an upper bound that tends to zero as $m \rightarrow \infty$, regardless of $p_i$ and $q_i$. It seems that one could also extend his result to allow inessential summands, as long as $m \rightarrow \infty$ for the number $m$ of lens space summands. However, $m$ would have be quite large to obtain a small isosystolic constant from asymptotic results of this type.

As of now, the best known isosystolic constant is still quite large for many of these remaining manifolds, perhaps as large as Gromov's universal constant $C_3 \approx$ 32 billion in some cases. In order to have a reasonably small isosystolic constant for them, it would be essential to study odd-order lens spaces closely.




\bibliography{SysZ2}
\bibliographystyle{amsalpha}

\end{document}